\documentclass[reqno]{amsart}
\usepackage{tkz-euclide}
\usepackage{hyperref}
\usepackage{amsmath}
\usepackage{amssymb} 
\usepackage{mathrsfs}
\usepackage{graphicx}
\usepackage{tikz}
\usetikzlibrary{calc}
\usepackage{xcolor}  
\usetikzlibrary{arrows.meta,snakes} 
\usepackage{amsthm}
\usepackage{float}
\usepackage{enumitem}
\usepackage[english]{babel}
\usepackage[font=footnotesize, labelfont=bf]{caption}

\definecolor{trueblue}{rgb}{0.0, 0.45, 0.81}

\newcommand{\EEE}{\color{black}}

\newcommand{\dx}{\, {\rm d}x}
\newcommand{\e}{\varepsilon}
\DeclareMathOperator{\conv}{conv}
\DeclareMathOperator{\dist}{dist}

\newcommand{\eg}{{\it e.g.}, }
\newcommand{\ie}{{\it i.e.}, }

\theoremstyle{plain}
\begingroup
\newtheorem{theorem}{Theorem}[section]
\newtheorem{lemma}[theorem]{Lemma}
\newtheorem{proposition}[theorem]{Proposition}

\endgroup

\numberwithin{equation}{section}
\newcommand{\N}{\mathbb{N}}
\newcommand{\Z}{\mathbb{Z}}
\newcommand{\Q}{\mathbb{Q}}
\newcommand{\R}{\mathbb{R}}

\newcommand{\NN}{\mathcal{N}}

\renewcommand{\S}{\mathbb{S}}
\renewcommand{\L}{\mathcal{L}}
\renewcommand{\H}{\mathcal{H}}
\newcommand{\dH}{\, {\rm d}\H^1}
\newcommand{\T}{\mathcal{T}}
\newcommand{\SF}{\mathcal{SF}}
\newcommand{\tb}{\hat{e}}
\newcommand{\x}{{\times}}
\newcommand{\defas}{:=}
\newcommand{\wto}{\rightharpoonup}
\newcommand{\wsto}{\overset{*}{\wto}}
\newcommand{\js}[1]{J_{#1}}
\newcommand{\cnu}{12}
\newcommand{\sinof}[1]{\sin\big(#1\big)}
\newcommand{\cosof}[1]{\cos\big(#1\big)}
\newcommand{\wcont}{\subset\subset}
\newcommand{\mres}{\mathbin{\vrule height 1.6ex depth 0pt width 0.13ex\vrule height 0.13ex depth 0pt width 1.3ex}}

\newenvironment{step}[1]{\textbf{Step #1.}}{}

\theoremstyle{definition}
\begingroup

\endgroup
\theoremstyle{remark}
\newtheorem{remark}[theorem]{Remark}
\newcommand{\pos}{\mathrm{pos}}
\renewcommand{\neg}{\mathrm{neg}}
\renewcommand{\tilde}{\widetilde}
\newcommand{\sm}{\setminus}
\newcommand{\cf}{{\it cf.\ }}
\DeclareMathOperator{\argmin}{argmin}
\DeclareMathOperator{\diam}{diam}
\DeclareMathOperator{\interior}{int}
\DeclareMathOperator{\sign}{sign}
\renewcommand{\d}{\, \mathrm{d}}

\usepackage[normalem]{ulem}

\begin{document}

\title[The antiferromagnetic $XY$ model on the triangular lattice]{The antiferromagnetic $XY$ model on the triangular lattice: Chirality transitions at the surface scaling}

\author{Annika Bach}
\address[Annika Bach]{TU M\"unchen, Germany}
\email{annika.bach@ma.tum.de}

\author{Marco Cicalese}
\address[Marco Cicalese]{TU M\"unchen, Germany}
\email{cicalese@ma.tum.de}

\author{Leonard Kreutz}
\address[Leonard Kreutz]{WWU M\"unster, Germany}
\email{lkreutz@uni-muenster.de}

\author{Gianluca Orlando}
\address[Gianluca Orlando]{TU M\"unchen, Germany}
\email{orlando@ma.tum.de}



\begin{abstract}
We study the discrete-to-continuum variational limit of the antiferromagnetic $XY$ model on the two-dimensional triangular lattice. The system is fully frustrated and displays two families of ground states distinguished by the chirality of the spin field. We compute the $\Gamma$-limit of the energy in a regime which detects chirality transitions on one-dimensional interfaces between the two admissible chirality phases.  
\end{abstract}

\maketitle

\setcounter{tocdepth}{1}

\maketitle

\section{Introduction}\label{section:introduction}

Ordering problems in magnetism have been extensively studied by both the physics and the mathematics communities. Researchers have been attracted by the rich phase diagrams and critical behaviors of magnetic models which are often the result of difficult-to-detect optimization effects taking place at several energy and length scales. The reason for such a complex behavior can be traced back to the presence of many competing mechanisms which give rise to frustration. {\em Frustration} in the context of spin systems (here, as it is customary in the statistical mechanics literature, we will often refer to magnets as to spins) refers to the situation where spins cannot find an orientation that simultaneously minimizes all the pairwise exchange interactions. Such interactions are said to be ferromagnetic or antiferromagnetic if they favour alignment or antialignment, respectively. Often frustration occurs in those systems where spins are subject to conflicting short range ferromagnetic and long range antiferromagnetic interactions, as when modulated phases appear (see, \eg the expository paper \cite{seuand}). For antiferromagnetic lattice systems, that is systems of lattice spins subject to only antiferromagnetic interactions, frustration can also stem from the relative spatial arrangement of spins induced by the geometry of the lattice. In this case frustration is often referred to as geometric frustration. As a consequence of geometric frustration magnetic compounds show complex geometric patterns that induce often unexpected effects whose understanding is one of the primary subjects in statistical and condensed matter physics as it can help to better explain the nature of phase transitions in magnetic materials~\cite{Die, LJNL, MS}. From a mathematical perspective, several interesting questions can be addressed. In this paper we are interested in the variational coarse graining of the system, in the line of what is by now addressed to as the ``discrete-to-continuum variational analysis of discrete systems''. Within this line of investigation the analysis of spin systems turns out to be a difficult nonlinear optimization problem requiring the combination of several methods ranging from simple discrete optimization procedures to sophisticated techniques in geometric measure theory and the calculus of variations. While models where frustration is induced by the competition of ferromagnetic/antiferromagnetic interactions have been already studied from a variational perspective (see, \eg~\cite{AliBraCic, GiuLebLie, GiuLieSer-11, GiuSer, CicSol, BraCic, SciVal, CicForOrl, DanRun}), what we present here is the first discrete-to-continuum result for a geometrically frustrated system.

We carry out the discrete-to-continuum variational analysis (at zero temperature) of a geometrically frustrated spin model in a specific energetic regime and we characterize the effective behavior of its low-energy states, that is states that can deviate from the global minimizers (ground states) by a certain small amount of energy. More precisely we consider a 2-dimensional nearest-neighbors antiferromagnetic planar spin model on the triangular lattice, \cf \cite[Chapter~1]{Die}. Despite being considered one of the most elementary geometrically frustrated spin models, its variational analysis turns out to be quite a delicate task. More in detail, we let $\e > 0$ be a small parameter and we consider the triangular lattice $\L_\e$ with spacing $\e$ (see Subsection~\ref{subsec:lattice} for the precise definition). To every spin field $u \colon \L_\e \to \S^1$ we associate the energy 
\begin{equation} \label{intro:AFXY energy}
    \sum_{\substack{\e \sigma, \e \sigma' \in \L_\e \\ |\sigma - \sigma'| = 1}} \langle u(\e \sigma) , u(\e \sigma') \rangle \, ,
\end{equation}
where $\langle \cdot , \cdot \rangle$ denotes the scalar product. (Below, the energy will be restricted to bounded regions in the plane.) This model is antiferromagnetic since the interaction energy between two neighboring spins is minimized by two opposite vectors. Such an order in the magnetic alignment, also known as antiferromagnetic order, is frustrated by the geometry of the triangular lattice, which inhibits a configuration where each pair of neighboring spins are opposite or, equivalently, where each interaction is minimized. This suggests that the antiferromagnetic $XY$ model depends substantially on the geometry of the lattice, which affects the structure of the ground states, the choice of the relevant variables and of the energy scalings. Notice, for example, that on a square lattice the system would not be frustrated, as opposite vectors distributed in a checkerboard structure minimize each interaction. In fact, on the square lattice a straightforward change of variable allows one to recast the antiferromagnetic $XY$ model into the ferromagnetic $XY$ model~\cite[Remark~3]{AliCic}, which is driven by an energy with neighboring interactions $- \langle u(\e \sigma), u(\e \sigma')\rangle$. The latter model has been thoroughly investigated in the last decade both on the square lattice~\cite{AliCic, AliCicPon, AliDLGarPon, CicOrlRuf1, CicOrlRuf2} and on the triangular lattice~\cite{CanSeg,DL}. Independently of the geometry of the lattice, it has been proved that spin fields that deviate from the ground states by an amount of energy which diverges logarithmically as $\e$ vanishes form of topological charges (vortex-like singularities of the spin field as those arising in the Ginzburg-Landau model~\cite{BBH, SS}), when subject to boundary conditions or external magnetic fields.

We now come back to our model~\eqref{intro:AFXY energy}. In order to identify the relevant variable of the system, we first need to characterize the ground states of the antiferromagnetic $XY$ system in~\eqref{intro:AFXY energy}. To this end it is convenient to rearrange the indices of the sum in~\eqref{intro:AFXY energy} and to recast the energy as a sum over all triangular plaquettes $T$ with vertices $\e i, \e j, \e k \in \L_\e$
\begin{equation} \label{intro:recasting the energy}
    \sum_{T} \big( \langle u(\e i),  u(\e j) \rangle + \langle u(\e j) , u(\e k)\rangle  +  \langle u(\e k),  u(\e i) \rangle  \big) = \frac{1}{2} \sum_{T} \big( |u(\e i) + u(\e j) + u(\e k) |^2 - 3 \big)\, .
\end{equation}
In each triangle $T$ the energy is minimized (and is equal to $-\frac{3}{2}$) if and only if $u(\e i) + u(\e j) + u(\e k) = 0$, namely, when the vectors of a triple $(u(\e i),u(\e j),u(\e k))$ point at the vertices of an equilateral triangle. By the $\S^1$-symmetry, every rotation of a minimizing triple $(u(\e i),u(\e j),u(\e k))$ is minimizing, too. The ground states in this model feature an additional symmetry, usually referred to as $\Z_2$-symmetry: triple obtained by from a minimizing triple via a permutation of negative sign as $(u(\e i),u(\e k),u(\e j))$ is also minimizing. This determines two families of ground states, \ie spin fields for which the energy is minimized in each plaquette, see Figure~\ref{fig:ground states}. These two families can be distinguished through the {\em chirality}, a scalar which quantifies the handedness of a certain spin structure. To define the chirality of a spin field $u$ in a triangle $T$, we need a consistent ordering of its vertices $\e i$, $\e j$, $\e k$. We assume that $\e i \in \L^1_\e$, $\e j \in \L^2_\e$, $\e k \in \L^3_\e$, where $\L^1_\e$, $\L^2_\e$, $\L^3_\e$ are the sublattices as in Figure~\ref{fig:ground states}, and we set (see~\eqref{def: chirality} for the precise definition)
\begin{equation*}
    \chi(u,T) = \frac{2}{3\sqrt{3}}\big(u(\e i)\x u(\e j)+u(\e j)\x u(\e k)+u(\e k)\x u(\e i) \big)  \in [-1,1]  \, ,
\end{equation*}
where the symbol $\times$ stems for the cross product. We denote by $\chi(u) \in L^\infty(\R^2)$ the function equal to $\chi(u,T)$ on the interior of each plaquette $T$. The ground states are exactly those configurations $u$ that satisfy either $\chi(u) \equiv 1$ or $\chi(u) \equiv - 1$, \cf Remark~\ref{rem:properties:chirality}.

\begin{figure}[H]
 \includegraphics{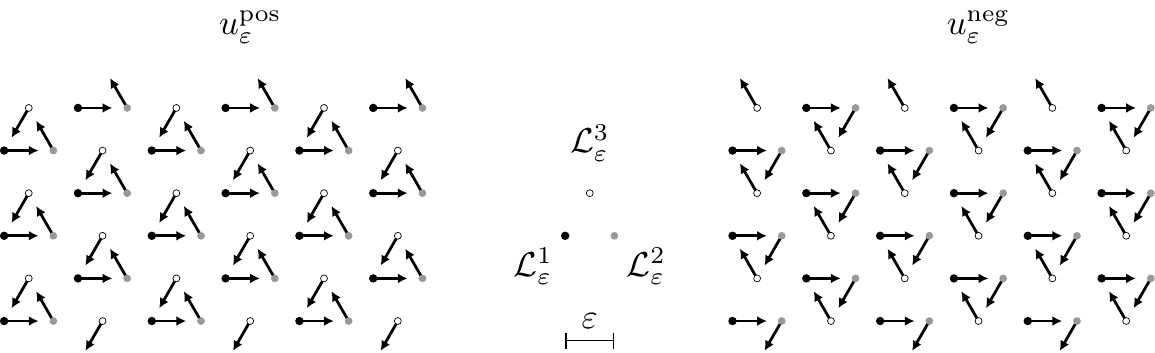}
\caption{A ground state $u_\e^\pos$ with positive chirality and a ground state $u_\e^\neg$ with negative chirality. Any other ground state of the system is obtained by composing one of these two configurations with a constant rotation. In the center: three points of the sublattices $\L^1_\e$, $\L^2_\e$, and $\L^3_\e$ in black, gray, and white, respectively.}
\label{fig:ground states}
\end{figure}

In this paper we analyze the energy regime at which the two families of ground states coexist and at the same time the energy of the system concentrates at the interface between the two chiral phases $\{\chi=1\}$ and $\{\chi=-1\}$. We fix $\Omega \subset \R^2$ open, bounded, and with Lipschitz boundary and we consider the energy~\eqref{intro:recasting the energy} restricted to $\Omega$, \ie computed only on plaquettes of $\L_\e$ contained in $\Omega$. We refer the energy to its minimum by removing the energy of the ground states ($-\frac{3}{2}$ for each plaquette) and we divide it by the number of lattice points in $\Omega$ (of order $1/\e^{2}$). We obtain (up to a multiplicative constant) the energy per particle given by
\begin{equation*}
    E_\e(u) = \sum_{T \subset \Omega} \e^2 |u(\e i) + u(\e j) + u(\e k) |^2.
\end{equation*}
We are interested to the asymptotic behavior of the energy above as $\e \to 0$ on sequences of spin fields $u_\e \colon \L_\e \to \S^1$ that can deviate from ground states yet satisfying a bound $E_\e(u_\e) \leq C \e$. To this end we define the energy $F_\e(u) := \frac{1}{\e} E_\e(u)$ and study sequences of spin fields with equibounded $F_{\e}$ energy. Due to the $\S^1$-symmetry, the energy at this regime cannot distinguish ground states with the same chirality, so that the relevant order parameter of the model is, in fact, not the spin field but its chirality: in Proposition~\ref{prop:compactness} we prove that a sequence $(u_\e)$ satisfying $F_\e(u_\e) \leq C$ admits a subsequence (not relabeled) such that $\chi(u_\e) \to \chi$ strongly in $L^1(\Omega)$ for some $\chi \in BV(\Omega;\{-1,1\})$, \ie the admissible chiralities in the continuum limit are~$-1$ and $1$ and the chirality phases $\{\chi = -1\}$ and $\{\chi = 1\}$ have finite perimeter in $\Omega$. This  suggests that the model shares similarities with systems having finitely many phases, such as Ising models~\cite{CafDLL05, AliBraCic, Pre} or Potts models~\cite{CicOrlRuf3}. However, a crucial difference consists in the fact that in our case the variable that shows a phase transition is not the spin variable itself, but the chirality, which depends on the spin field in a nonlinear way. This is a source of difficulties that will be explained below. 

To describe the asymptotic behavior of the system it is convenient to introduce the functionals depending only on functions $\chi \in L^1(\Omega)$ defined (with a slight abuse of notation) by $F_\e(\chi) := \inf \{ F_\e(u) \colon \ u \colon \L_\e \to \S^1 \text{ such that } \chi = \chi(u,T) \text{ on every } T \subset \Omega \}$ (equal to $+\infty$ if $\chi$ is not the chirality of a spin field). The main result in this paper is Theorem~\ref{maintheorem}, where we prove that the $\Gamma$-limit of $F_\e$ with respect to the $L^1$-convergence is an anisotropic surface energy given by 
\begin{equation*}
    F(\chi) = \int_{\js{\chi}} \varphi(\nu_\chi) \, \mathrm{d}\mathcal{H}^1 \quad \text{for } \chi \in BV(\Omega;\{-1,1\})\, ,
\end{equation*}
extended to $+\infty$ otherwise in $L^1(\Omega)$, where $J_\chi$ is the interface between $\{\chi = -1\}$ and $\{\chi = 1\}$ and $\nu_\chi$ is the normal to $J_\chi$. The density $\varphi$ is given by the following asymptotic formula
\begin{equation} \label{intro:cell formula}
    \varphi(\nu) = \lim_{\varepsilon \to 0} \min\left\{ F_\varepsilon(u,Q^\nu) \colon u = u_\e^\pos \text{ on } \partial_\varepsilon^+ Q^\nu\, ,\ u=u_\e^\neg \ \text{on}\ \partial_\e^- Q^\nu\right\}\, ,
\end{equation}
where $Q^\nu$ is the square with one side orthogonal to $\nu$, $u_\e^\pos$ and $u_\e^\neg$ are the ground states depicted in Figure~\ref{fig:ground states}, and $\partial_\e^\pm  Q^\nu$ are a discrete version of the top/bottom parts of $\partial Q^\nu$. 
Asymptotic formulas like~\eqref{intro:cell formula} are common in discrete-to-continuum variational analyses and are often used to represent $\Gamma$-limits of discrete energies~\cite{AG,BraCic, BraPiat, BraKre, BacBraCic, FriKreSch}. However, proving an asymptotic lower bound with the density~\eqref{intro:cell formula} for this model requires additional care and is the technically most demanding contribution of this paper. We conclude this introduction by describing the main difficulties that arise in the proof. 

Via a classical blow-up argument (see Proposition~\ref{prop:lowerbound}) we obtain an asymptotic lower bound with the surface density
\begin{equation} \label{intro:psi}
    \psi(\nu) = \inf\big\{\liminf_{\e\to 0}F_\e(u_\e,Q^\nu)\colon\chi(u_\e)\to\chi_\nu\ \text{in}\ L^1(Q^\nu)\big\} \, ,
\end{equation}
where $\chi_\nu$ is the pure-jump function which takes the values $\chi_\nu(x)=\pm 1$ for $\pm \langle x, \nu \rangle > 0$. Hence, the proof of the asymptotic lower bound boils down to the proof of the inequality $\psi(\nu) \geq \varphi(\nu)$. To obtain the latter inequality, we need to modify sequences $(u_\e)$ with $\chi(u_\e)\to\chi_\nu$ in $L^1(Q^\nu)$ without increasing their energy in such a way that they attain the boundary conditions required in~\eqref{intro:cell formula}. A common approach to deal with this modification consists in selecting (via a well-known slicing/averaging argument due to De Giorgi) a low-energy frame contained in $Q^\nu$ and close to $\partial Q^\nu$ where the sequence can be modified using a cut-off function that interpolates to the boundary values. In our problem, instead, a cut-off modification of $\chi(u_\e)$ may generate a sequence of functions that are not chiralities of spin fields (and thus have infinite energy $F_\e$). Consequently, we have to operate directly on the sequence~$(u_\e)$, on whose convergence we have no information due to the invariance of the system under rotation of the spin field (the $\S^1$-symmetry discussed above).
We turn however the $\S^1$-symmetry to our advantage to define the needed modification. Inside a one-dimensional slice of~$\L_\e$, a spin field close to a ground state in one triangle can be slowly rotated to reach any other ground state with the same chirality by paying an amount of energy proportional to the energy in the starting triangle (see Lemma~\ref{lemma:1d interpolation}). This one-dimensional construction can then be reproduced in the whole $Q^{\nu}$ starting from triangles in a low-energy frame close to $\partial Q^\nu$ in such a way that the modified spin field attains the fixed ground states $u_\e^\pos$ and $u_\e^\neg$ at the (discrete) boundary. However, for this procedure to be successful, the usual slicing/averaging method to find a low-energy frame close to $\partial Q^\nu$ is not enough. 
We need to improve it and to find a frame with a better (smaller) energy bound. To this end, we proceed as follows. In Lemma \ref{lem:independence} we show that 
$\psi(\nu)$ can be equivalently defined using in place of $Q^\nu$ any rectangle coinciding with $Q^\nu$ along the interface, but with arbitrarily small height (similar results appeared in different contexts, \eg\cite{CafDLL01, CafDLL05, ConFonLeo, ChaGolNov, ConGarMas, ConSch, FriKreSch}). Hence the energy of any sequence $(u_\e)$ admissible for~\eqref{intro:psi} concentrates arbitrarily close to the jump set of $\chi_\nu$, \ie the interface $\{\langle x,\nu\rangle=0\}$. With this result at hand, in Lemma~\ref{lemma:choosing strip} we can apply the averaging method with the advantage of knowing that in most of the space the total energy is going to vanish, thus finally deducing the existence of a frame close to $\partial Q^\nu$ with the wished (small enough) energy bound. 
Even at this point, to reproduce the one-dimensional interpolation along this frame requires additional care. In fact, to conclude the argument one still needs to prove that the winding number of the spin field in the low-energy frame can be properly controlled (Step $3$ of Proposition \ref{prop:psi is phi}).  

\parskip1.3mm

\section{Setting of the problem and statement of the main result}\label{section:notation}

\subsection{General notation}
Throughout this paper $\Omega\subset\R^2$ is an open, bounded set with Lipschitz boundary. For every $A\subset\R^2$ measurable we denote by $|A|$ its $2$-dimensional Lebesgue measure. With $\H^1$ we indicate the $1$-dimensional Hausdorff measure in $\R^2$. Given two points $x,y \in \R^2$ we use the notation $[x;y] := \{ \lambda x + (1-\lambda) y \colon \lambda \in [0,1] \}$ for the segment joining $x$ and $y$.  The set $\S^1\defas\{\nu\in\R^2\colon|\nu|=1\}$ is the set of all 2-dimensional unit vectors. For every such vector $\nu=(\nu_1,\nu_2)\in\S^1$ we denote by $\nu^\perp\defas(-\nu_2,\nu_1)\in\S^1$ the unit vector orthogonal to $\nu$ obtained by rotating $\nu$ counterclockwise by $\pi/2$. Given $v,w \in \S^1$ we denote by $\langle v, w\rangle$ their scalar product and by $v \x w =-\langle v, w^\perp \rangle$ their cross product. We denote by $\iota$ the imaginary unit in the complex plane. It will be often convenient to write vectors in $\S^1$ as $\exp(\iota \theta)$, $\theta \in \R$. We denote by $R_{\ell,h}^\nu$ the rectangle of length $\ell>0$ and height $h>0$ with two sides orthogonal to $\nu \in \S^1$ given by 
\begin{equation*}
    R_{\ell,h}^\nu\defas\{x\in\R^2\colon |\langle x,\nu^\perp\rangle|<\ell/2 \, ,\ |\langle x,\nu\rangle|<h/2\}\, ,
    \end{equation*}
extending the definition to the case $\ell = +\infty$ by setting $R_{\infty,h}^\nu\defas\{x\in\R^2\colon |\langle x,\nu\rangle|<h/2\}$. Given $\rho>0$ we define the cube centered at the origin with side length $\rho$ and one face orthogonal to $\nu$ by $Q_\rho^\nu\defas R_{\rho,\rho}^\nu$. For $\rho=1$ we simply write $Q^\nu$ instead of $Q_1^\nu$. By $L^\nu\defas\{x\in\R^2\colon\langle x,\nu\rangle=0\}$ we denote the line orthogonal to $\nu$ passing through the origin, while $H^\nu_+\defas\{x\in\R^2\colon\langle x,\nu\rangle\geq 0\}$ and $H^\nu_-\defas\R^2\setminus H^\nu_+$ stand for the two half spaces separated by $L^\nu$. 
Given $x_0\in\R^2$ we set $Q_\rho^\nu(x_0)\defas x_0+Q_\rho^\nu$, $R_{\ell,h}^\nu(x_0)\defas x_0+R_{\ell,h}^\nu$, $L^\nu(x_0)\defas x_0+L^\nu$ and $H^\nu_\pm(x_0)\defas x_0+H^\nu_\pm$.

\subsection{Triangular lattices and discrete energies}\label{subsec:lattice}
 
In this paragraph we define the discrete energy functionals we consider in this paper. To this end we first define the triangular lattice $\L$. It is given by
\begin{align*}
\mathcal{L}:= \{z_1 \tb_1 + z_2 \tb_2\colon z_1,z_2 \in \mathbb{Z}\}\, ,
\end{align*}
with $\tb_1=(1,0)$, and $\tb_2=\frac{1}{2}(1,\sqrt{3})$. For later use, we find it convenient here to introduce $\tb_3\defas\tfrac{1}{2}(-1,\sqrt{3})$ as a further unit vector connecting points of $\L$ and to define three pairwise disjoint sublattices of $\L$, denoted by $\L^1$, $\L^2$, and $\L^3$ (see Figure~\ref{fig:ground states}), by
\begin{equation*} 
\L^1\defas\{z_1(\tb_1+\tb_2)+z_2(\tb_2+\tb_3)\colon z_1,z_2\in\Z\}\, ,\quad\L^2\defas\L^1+\tb_1\, ,\quad\L^3\defas\L^1+\tb_2\, .
\end{equation*}
Eventually, we define the family of triangles subordinated to the lattice $\L$ by setting
\begin{equation*}
\T(\R^2)\defas\big\{T=\conv\{i,j,k\}\colon i,j,k\in\L,\ |i-j|=|j-k|=|k-i|=1\big\}\, ,
\end{equation*}
where $\conv\{i,j,k\}$ denotes the closed convex hull of $i,j,k$. It is also convenient to introduce the families of upward/downward facing triangles 
\begin{equation*} 
    \T^\pm(\R^2) \defas \big\{T=\conv\{i,j,k\}\in \T(\R^2)  \colon i\in \L^1,j \in \L^2, k\in\L^3,\ \pm (j-i) \x (k-j) \x (i-k) > 0 \big\}\, .
\end{equation*}

For $\e>0$, we consider rescaled versions of $\L$ and $\T(\R^2)$ given by $\L_\e\defas\e\L$ and $\T_\e(\R^2)\defas\e\T(\R^2)$, $\T^\pm_\e(\R^2) \defas \e \T^\pm(\R^2)$. With this notation every $T\in\T_\e(\R^2)$ has vertices $\e i,\e j,\e k\in\L_\e$. The same notation applies to the sublattices, namely $\L_\e^\alpha\defas\e\L^\alpha$ for $\alpha\in\{1,2,3\}$. Given a Borel set $A\subset\R^2$ we denote by $\T_\e(A)\defas\{T\in\T_\e(\R^2)\colon T\subset A\}$ the subfamily of triangles contained in~$A$. Eventually, we introduce the set of admissible configurations as the set of all \textit{spin fields}
\begin{equation*} 
\SF_\e\defas\{u:\L_\e\to\S^1\}\, .
\end{equation*}
In the case $\e=1$ we set $\SF\defas \SF_1$. For $u\in\SF_\e$ we now define the discrete energies $F_\e(u)$ as follows: for every $T\in\T_\e(\R^2)$ we set
\begin{equation*}
F_\e(u,T)\defas\e|u(\e i)+u(\e j)+u(\e k)|^2\, ,
\end{equation*}
and we extend the energy to any Borel set $A\subset\R^2$ by setting
\begin{equation*} 
F_\e(u,A)\defas\sum_{T\in\T_\e(A)}F_\e(u,T)\,.
\end{equation*}
If $A=\Omega$ we omit the dependence on the set and write $F_\e(u):=F_\e(u,\Omega)$. 
\subsection{Chirality}
In this section we introduce the relevant order parameter to analyze the asymptotic behavior of $F_\e$, namely the chirality $\chi$. More in detail, given $u\in\SF_\e$ and $T=\conv\{\e i,\e j,\e k\}\in\T_\e(\R^2)$ with $i\in\L^1$, $j\in\L^2$ and $k\in\L^3$ we set
\begin{equation}\label{def: chirality}
\chi(u,T)\defas\frac{2}{3\sqrt{3}}\big(u(\e i)\x u(\e j)+u(\e j)\x u(\e k)+u(\e k)\x u(\e i) \big)\, .
\end{equation}
Moreover, we define $\chi(u)\colon \Omega\to\R$ by setting $\chi(u)(x)\defas \chi_\e(u,T)$ if $x\in\interior T$. Given $u\in\SF_\e$ and $T=\conv\{\e i,\e j,\e k\}\in\T_\e(\R^2)$ it is sometimes convenient to rewrite $\chi_\e(u,T)$ and $F_\e(u,T)$ in terms of the angular lift of $u$. More precisely, let $\theta(\e i),\theta(\e j),\theta(\e k)\in\R$ be such that $u(\e\alpha)=\exp(\iota\theta(\e \alpha))$, $\alpha\in\{i,j,k\}$. Then
\begin{align}
\chi(u,T) &=\frac{2}{3\sqrt{3}}\Big(\sinof{\theta(\e j)-\theta(\e i)}+\sinof{\theta(\e k)-\theta(\e j)}+\sinof{\theta(\e i)-\theta(\e k)}\Big)\, ,\label{chi:angularlift}\\
F_\e(u,T) &=3\e+2\e\Big(\cosof{\theta(\e j)-\theta(\e i)}+\cosof{\theta(\e k)-\theta(\e j)}+\cosof{\theta(\e i)-\theta(\e k)}\Big)\, .\label{F:angularlift}
\end{align}
The next lemma is useful to relate the chirality and the energy in a triangle. 
\begin{lemma}\label{lem:chirality}
Let $f,g:[0,2\pi)\x[0,2\pi)\to\R$ be given by
\begin{align*}
f(\theta_1,\theta_2) &\defas\sin(\theta_1)+\sin(\theta_2-\theta_1)-\sin(\theta_2)\, ,\\
g(\theta_1,\theta_2) &\defas\cos(\theta_1)+\cos(\theta_2-\theta_1)+\cos(\theta_2)\, .
\end{align*}
Then $f$ and $g$ have the following properties:
\begin{enumerate}[label={\rm(\roman*)}]
\item $f(\theta_1,\theta_2)\in[-\tfrac{3\sqrt{3}}{2},\tfrac{3\sqrt{3}}{2}]$ for every $\theta_1,\theta_2\in[0,2\pi)$. Moreover, $f(\theta_1,\theta_2)\in\{-\tfrac{3\sqrt{3}}{2},\tfrac{3\sqrt{3}}{2}\}$ if and only if $g(\theta_1,\theta_2)=-\tfrac{3}{2}$.
\item $f(\theta_1,\theta_1)=f(\theta_1,0)=f(0,\theta_2)=0$ for every $\theta_1,\theta_2\in[0,2\pi)$. In addition, for every $\theta_2\in(0,2\pi)$ there holds $f(\, \cdot\, ,\theta_2)>0$ on $(0,\theta_2)$ and $f(\, \cdot\, ,\theta_2)<0$ on $(\theta_2,2\pi)$.
\end{enumerate}
\end{lemma}
\begin{proof}
Since there obviously holds $f(\theta_1,\theta_1)=f(\theta_1,0)=f(0,\theta_2)=0$, we only need to prove (i) and the second part of (ii). 
To prove (i) we show that $\min f=-\tfrac{3\sqrt{3}}{2}$ and $\max f=\tfrac{3\sqrt{3}}{2}$ and we relate minimizers and maximizers of $f$ to minimizers of $g$. 
To this end we start computing 
\begin{equation*}
\nabla f(\theta_1,\theta_2)=
\begin{pmatrix}
\cos(\theta_1)-\cos(\theta_2-\theta_1)\\
\cos(\theta_2-\theta_1)-\cos(\theta_2)
\end{pmatrix}\, 
\quad\text{and}\quad
\nabla g(\theta_1,\theta_2)=
\begin{pmatrix}
-\sin(\theta_1)+\sin(\theta_2-\theta_1)\\
-\sin(\theta_2-\theta_1)-\sin(\theta_2)
\end{pmatrix}\, .
\end{equation*}
A direct calculation shows that $\nabla f(\theta_1,\theta_2)=0$ for some $(\theta_1,\theta_2)\in(0,2\pi)\x(0,2\pi)$ if and only if  
\begin{equation}\label{cond:criticalpoints}
\theta_1=\frac{\theta_2}{2}+z_1\pi\quad\text{and}\quad \theta_2=\frac{\theta_1}{2}+z_2\pi\, ,\quad \text{for some } z_1,z_2\in\{0,1\} \, .
\end{equation}
For $(\theta_1,\theta_2)\in(0,2\pi)\x(0,2\pi)$  this can only be satisfied if
\begin{equation} \label{eq:wells}
(\theta_1,\theta_2)=(\tfrac{2\pi}{3},\tfrac{4\pi}{3})\quad\text{or}\quad(\theta_1,\theta_2)=(\tfrac{4\pi}{3},\tfrac{2\pi}{3})\, .
\end{equation}
Then, since $f=0$ on the boundary of $[0,2\pi)\x[0,2\pi)$, we deduce that
\begin{equation*}
\min_{[0,2\pi)\x [0,2\pi)}f=f\Big(\Big(\frac{4\pi}{3},\frac{2\pi}{3}\Big)\Big)=-\frac{3\sqrt{3}}{2}\quad\text{and}\quad\max_{[0,2\pi)\x[0,2\pi)}f=f\Big(\Big(\frac{2\pi}{3},\frac{4\pi}{3}\Big)\Big)=\frac{3\sqrt{3}}{2}\, .
\end{equation*}
Moreover, $g((\tfrac{2\pi}{3},\tfrac{4\pi}{3}))=g((\tfrac{4\pi}{3},\tfrac{2\pi}{3}))=-\tfrac{3}{2}$, which shows one direction of the second part of (i). To prove the opposite direction, let us assume that $(\overline{\theta}_1,\overline{\theta}_2)\in(0,2\pi)\x(0,2\pi)$ is such that $g((\overline{\theta}_1,\overline{\theta}_2)=\min g$. Then necessarily $\nabla g(\overline{\theta}_1,\overline{\theta}_2)=0$, from which we deduce that $(\overline{\theta}_1,\overline{\theta}_2)$ must satisfy~\eqref{cond:criticalpoints} (the possibility that $\overline{\theta}_1 = \pi$ or $\overline{\theta}_2 = \pi$ are ruled out by the fact that $g(\, \cdot\, ,\pi)=g(\pi,\, \cdot\, )=-1$). The pairs $(\overline{\theta}_1,\overline{\theta}_2)$ satisfying \eqref{cond:criticalpoints} are either $(\overline{\theta}_1,\overline{\theta}_2)=(\tfrac{2\pi}{3},\tfrac{4\pi}{3})$ or $(\overline{\theta}_1,\overline{\theta}_2)=(\tfrac{4\pi}{3},\tfrac{2\pi}{3})$ and in both cases it holds $g(\overline{\theta}_1,\overline{\theta}_2)=-\tfrac{3}{2}$. This yields that $\min g=-\tfrac{3}{2}$ and that the opposite direction of (i) holds, upon noticing that $g\geq-1$ on the boundary of $[0,2\pi)\x[0,2\pi)$. 
To complete the proof of (ii) let us fix $\theta_2\in(0,2\pi)$ and consider $f(\, \cdot\, ,\theta_2)$ as a function of $\theta_1$. Then \eqref{cond:criticalpoints} shows that $\tfrac{\partial f}{\partial \theta_1}(\theta_1,\theta_2)=0$ if and only if $\theta_1\in\{\theta_2^{\rm pos},\theta_2^{\rm neg}\}$, where
\begin{equation}\label{def:thetapn}
\theta_2^{\rm pos}\defas\frac{\theta_2}{2}\in(0,\theta_2)\, ,\quad \theta_2^{\rm neg}\defas\frac{\theta_2}{2}+\pi\in(\theta_2,2\pi)\, .
\end{equation}
Moreover, upon extending $f(\, \cdot\, ,\theta_2)$ to an open interval containing $(0,2\pi)$, we get 
\begin{equation*}
\frac{\partial f}{\partial \theta_1 }(0,\theta_2)=\frac{\partial f}{\partial \theta_1 }(2\pi,\theta_2)=1-\cos(\theta_2)>0\quad\text{and}\quad\frac{\partial f}{\partial \theta_1 }(\theta_2,\theta_2)=\cos(\theta_2)-1<0\, .
\end{equation*}
In particular, from the intermediate value theorem we deduce that $f(\, \cdot\, ,\theta_2)$ is strictly increasing on $(0,\theta_2^{\rm pos})$ and strictly decreasing on $(\theta_2^{\rm pos},\theta_2)$. Since in addition $f(0,\theta_2)=f(\theta_2,\theta_2)=0$ this implies that $f(\, \cdot\, ,\theta_2)>0$ on $(0,\theta_2)$. Arguing similarly on the intervals $(\theta_2,\theta_2^{\rm neg})$ and $(\theta_2^{\rm neg},2\pi)$ we obtain $f(\, \cdot\, ,\theta_2)<0$ on $(\theta_2,2\pi)$, which proves (ii).
\end{proof}

\begin{remark}\label{rem:properties:chirality}
    Using the expressions of $\chi(u,T)$ and $F_\e(u,T)$ in \eqref{chi:angularlift}--\eqref{F:angularlift} one can show that $\chi(u,T)\in[-1,1]$ and $\chi(u,T)\in\{-1,1\}$ if and only if $F_\e(u,T)=0$, \ie configurations that maximize or minimize $\chi(\, \cdot\, ,T)$ are at the same time minimizers for $F_\e(\, \cdot\, ,T)$. This follows from Lemma~\ref{lem:chirality}~(i) upon noticing that in~\eqref{chi:angularlift}--\eqref{F:angularlift} it is not restrictive to assume that $\theta(\e i)=0$, since both $\chi_\e$ and $F_\e$ are invariant under rotations in $u$. We observe that also a quantitative version of this property holds. Namely, a continuity argument shows that for every $\delta>0$ there exists $C_\delta>0$ such that for every $u\in\SF_\e$ and every $T\in\T_\e(\R^2)$ the following implication holds: 
    \begin{equation}\label{impl:delta}
    \chi(u,T)\in (-1+\delta,1-\delta)\quad\implies\quad F_\e(u,T)\geq C_\delta\e\, .
    \end{equation}
    \end{remark}
    \begin{remark}\label{rem:orientation}
        As a consequence of Lemma~\ref{lem:chirality}~(ii) one obtains the following characterization of the sign of the chirality. Let $\theta(\e j) \in [0,2\pi)$ be the angle between $u(\e i)$ and $u(\e j)$ and let $\theta(\e k) \in [0,2\pi)$ the angle between $u(\e i)$ and $u(\e k)$. Then $\chi(u,T)>0$ if and only if $\theta(\e j)<\theta(\e k)$ and and $\chi(u,T)< 0$ if and only if $\theta(\e j)>\theta(\e k)$. In other words, a positive chirality on $T=\conv\{\e i,\e j,\e k\}$ corresponds to a counterclockwise ordering of $(u(\e i), u(\e j), u(\e k))$ on $\S^1$, while a negative chirality corresponds to a clockwise ordering on $\S^1$.  
    \end{remark}

\begin{figure}[H]
    \includegraphics{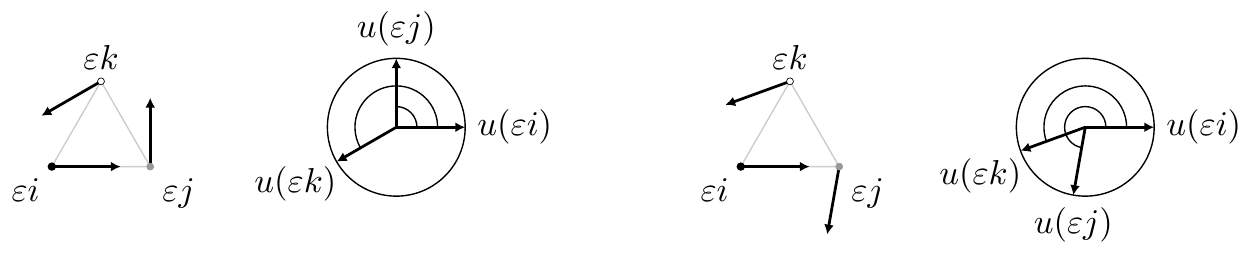}
\caption{On the left: a configuration of vectors with positive chirality which shows the criterion explained in Remark~\ref{rem:orientation}. On the right: negative chirality.}
\end{figure}
    
    
%
\subsection{Statement of the main result}   
Notice that $\chi(u)\in L^1(\Omega)$. We then extend $F_\e$ to $L^1(\Omega)$ by setting
\begin{align}\label{eq:extended energy}
F_\varepsilon(\chi) = \inf\{F_\varepsilon(u)\colon u\in\SF_\e \, ,\ \chi(u)=\chi\}\, ,
\end{align}
with the convention $\inf\emptyset=+\infty$.
\begin{remark}\label{rem:def:extenergy}
If $\chi\in L^1(\Omega)$ is such that $\chi=\chi(u)$ for some $u\in\SF_\e$, then the infimum in \eqref{eq:extended energy} is actually a minimum.
\end{remark}
To state the main theorem we need to introduce two ground states, that we name $u_\e^{\rm pos},u_\e^{\rm neg}\in\SF_\e$ which have a uniform chirality equal to $+1$ and $-1$, respectively. They are given by
\begin{equation*} 
u_\e^{\rm pos}(x)\defas
\begin{cases}
\exp(\iota 0) &\text{if}\ x\in\L_\e^1\, ,\\
\exp(\iota 2\pi/3) &\text{if}\ x\in\L_\e^2\, ,\\
\exp(\iota 4\pi/3) &\text{if}\ x\in\L_\e^3\, ,
\end{cases}
\qquad
u_\e^{\rm neg}(x)\defas
\begin{cases}
\exp(\iota 0) &\text{if}\ x\in\L_\e^1\, ,\\
\exp(\iota 4\pi/3) &\text{if}\ x\in\L_\e^2\, ,\\
\exp(\iota 2\pi/3) &\text{if}\ x\in\L_\e^3\, ,
\end{cases}
\end{equation*}
for every $x\in\L_\e$. We also set $u^\pos\defas u_1^\pos$, $u^\neg\defas u_1^\neg$. The ground states $u^\pos$ and $u^\neg$ will be used as boundary conditions on the discrete boundary of the square $Q^\nu$ given by 
\begin{align}\label{def:partialQeps}
    \partial_\e^\pm Q^\nu =\{ x \in \mathcal{L}_\varepsilon\colon \pm\langle\nu, x\rangle \geq  3\varepsilon\, ,\ \mathrm{dist}(x,\partial Q^\nu) \leq 3 \varepsilon\}\, .
    \end{align}
\begin{theorem}\label{maintheorem}
The energies $F_\e \colon L^1(\Omega)\to[0,+\infty]$ defined by \eqref{eq:extended energy} $\Gamma$-converge in the strong $L^1(\Omega)$-topology to the functional $F \colon L^1(\Omega)\to[0,+\infty]$ given by
\begin{equation}\label{def: limit energy}
F(\chi) := 
\begin{cases} 
\displaystyle\int_{\js{\chi}} \varphi(\nu_\chi) \d\mathcal{H}^1 &\text{if}\ \chi \in BV(\Omega;\{-1,1\})\, ,\\
+\infty &\text{otherwise in}\ L^1(\Omega)\, ,
\end{cases}
\end{equation}
where $\varphi\colon \S^1\to[0,+\infty)$ is defined by
\begin{equation}
\label{def: varphi}
\varphi(\nu)\defas\lim_{\varepsilon \to 0} \min\left\{ F_\varepsilon(u,Q^\nu) \colon u = u_\e^{\rm pos}\text{ on } \partial_\varepsilon^+ Q^\nu ,\ u=u_\e^{\rm neg}\ \text{on}\ \partial_\e^- Q^\nu\right\}\, .
\end{equation}
\end{theorem}

The proof of Theorem~\ref{maintheorem} will be carried out in Sections~\ref{section:lowerbound} and~\ref{section:upperbound}, where we prove separately the asymptotic lower bound (Proposition~\ref{prop:lowerbound}) and the asymptotic upper bound (Proposition~\ref{prop:upperbound}), respectively. 

\begin{remark} \label{rem: well-defined} By standard arguments in the analysis of asymptotic cell formulas (see e.g. \cite[Proposition 4.6]{AliCicSig}) one can show that the limit in \eqref{def: varphi} actually exists, so that $\varphi$ is well defined. Note that, by the symmetries of the interaction energies, there holds $\varphi(-\nu)=\varphi(\nu)$. Moreover, one can show (cf.~\cite[Proposition 4.7]{AliCicSig}) that the one-homogeneous extension of $\varphi$ is convex, hence continuous. 
\end{remark}

\EEE
\begin{remark} By a scaling argument we note that for all $\rho >0$ there holds
\begin{align}\label{eq: varphi rho}
\varphi(\nu)  = \lim_{\varepsilon \to 0}\frac{1}{\rho} \min\big\{ F_\varepsilon(u,Q^\nu_\rho) \colon u = u^{\rm pos}_\varepsilon \text{ on } \partial_\varepsilon^+ Q^\nu_\rho\, ,\ u=u_\e^{\rm neg}\ \text{on}\ \partial_\e^- Q_\rho^\nu\big\}\, ,
\end{align}
where $\partial_\e^\pm Q_\rho^\nu$ are defined according to \eqref{def:partialQeps} with $Q_\rho^\nu$ in place of $Q^\nu$.
\end{remark}
\section{Compactness}
\begin{proposition}\label{prop:compactness}
Let $(u_\e)$ be a sequence of spin fields $u_\e\in\SF_\e$ satisfying
\begin{equation}\label{comp:uniformbound}
\sup_\e F_\e(u_\e)<+\infty\, .
\end{equation}
Then there exists $\chi\in BV(\Omega;\{-1,1\})$ such that up to subsequences $\chi(u_\e)\to \chi$ in $L^1(\Omega)$.
\end{proposition}
To prove Proposition \ref{prop:compactness} we first estimate from below the energy of a spin field $u$ on two neighboring triangles where $\chi(u)$ changes sign. Given a triangle $T\in\T_{\e}(\R^{2})$ we introduce the class $\NN_\e(T)$ of its neighboring triangles, namely those triangles in $\T_{\e}(\R^{2})$ that share a side with $T$. More precisely, we define
\begin{equation}\label{def:neighbors}
\NN_\e(T)\defas\{T'\in\T_\e(\R^2)\colon \H^1(T\cap T')=\e\}\, .
\end{equation}
\begin{lemma}\label{lemma:energy of opposite chirality}
Let $u\in\SF_\e$ and suppose that $T^{\rm pos}, T^{\rm neg}\in\T_\e(\R^2)$ with $T^\neg\in\NN_\e(T^\pos)$ are such that $\chi(u,T^{\rm pos})\geq 0$ and $\chi(u,T^{\rm neg})\leq 0$. Then $F_\e(u,T^{\rm pos}\cup T^{\rm neg})\geq\tfrac{5}{3}\e$. 
\end{lemma}
\begin{proof}
It is not restrictive to assume that $T^{\rm pos}=\conv\{\e i,\e j,\e k\}$ and $T^{\rm neg}=\conv\{\e i,\e j',\e k\}$ with $i\in\L^1$, $j,j'\in\L^2$ and $k\in\L^3$. Moreover, we can assume that $u(\e i)=\tb_1$, that is, $\theta(\e i)=0$ according to the notation in \eqref{chi:angularlift}--\eqref{F:angularlift}. Then, using the function $g\colon [0,2\pi)\times[0,2\pi)\to\R$ defined in Lemma \ref{lem:chirality}, we can rewrite $F_\e(u,T^{\rm pos}\cup T^{\rm neg})$ as
\begin{equation*}
F_\e(u,T^{\rm pos}\cup T^{\rm neg})=6\e+2\e\Big(g\big(\theta(\e j),\theta(\e k)\big)+g\big(\theta(\e j'),\theta(\e k)\big)\Big)\, .
\end{equation*}
Moreover, thanks to Lemma \ref{lem:chirality} (ii) the chirality constraint reads $0\leq\theta(\e j)\leq\theta(\e k)\leq\theta(\e j')$. Thus, the statement is proved if we show that for all $\theta_1$, $\theta_2$, $\theta_3\in[0,2\pi)$ with $0\leq\theta_1\leq\theta_2\leq\theta_3$ there holds
\begin{equation}\label{est:g}
6+2\big(g(\theta_1,\theta_2)+g(\theta_3,\theta_2)\big)\geq\frac{5}{3}\, .
\end{equation}
We first observe that \eqref{est:g} trivially holds if $\theta_2=0$ or $\theta_2=\pi$.
Indeed, if $\theta_2=0$, then also $\theta_1=0$, hence $g(\theta_1,\theta_2)+g(\theta_3,\theta_2)=4+2\cos(\theta_3)\geq 2$, thus \eqref{est:g} is satisfied. If, instead, $\theta_2=\pi$, then a direct computation shows that $g(\theta_1,\theta_2)+g(\theta_3,\theta_2)=  -2$ for every $\theta_1,\theta_3\in[0,2\pi)$, which directly gives \eqref{est:g}.

Suppose now that $\theta_2\in(0,2\pi)\setminus\{\pi\}$ and let us minimize $g(\, \cdot\, ,\theta_2)$ on the two intervals $[0,\theta_2]$ and $[\theta_2,2\pi)$. As in the proof of Lemma \ref{lem:chirality} we obtain that $\frac{\partial g}{\partial_{\theta_1}}(\theta_1,\theta_2)=0$ if and only if $\theta_1\in\{\theta_2^{\rm pos},\theta_2^{\rm neg}\}$ with $\theta_2^{\rm pos}$, $\theta_2^{\rm neg}$ as in \eqref{def:thetapn}. Moreover, we have
\begin{equation}\label{cond:2nd:derivative}
\frac{\partial^2 g}{\partial\theta_1^{2}}(\theta_2^{\rm pos},\theta_2)=-2\cos\Big(\frac{\theta_2}{2}\Big)\quad\text{and}\quad \frac{\partial^2 g}{\partial\theta_1^{2}}(\theta_2^{\rm neg},\theta_2)=2\cos\Big(\frac{\theta_2}{2}\Big)\, .
\end{equation}
Thus, either $\theta_2^{\rm pos}\in(0,\theta_2)$ or $\theta_2^{\rm neg}\in(\theta_2,2\pi)$ is a minimizer for $g(\, \cdot\, ,\theta_2)$, depending on wether $\theta_2\in (0,\pi)$ or $\theta_2\in(\pi,2\pi)$. Suppose first that $\theta_2\in (\pi,2\pi)$. Then \eqref{cond:2nd:derivative} implies that $g(\, \cdot\, ,\theta_2)$ is minimized in $[0,\theta_2)$ by $\theta_2^{\rm pos}$, while in $[\theta_2,2\pi)$ it attains its minimum on the boundary, that is at $\theta_2$. This yields
\begin{equation}\label{est:g:1}
g(\theta_1,\theta_2)+g(\theta_3,\theta_2)\geq g(\theta_2^{\rm pos},\theta_2)+g(\theta_2,\theta_2)=2\cos\Big(\frac{\theta_2}{2}\Big)+3\cos(\theta_2)+1\, ,
\end{equation}
for every $\theta_1\in[0,\theta_2]$ and $\theta_3\in[\theta_2,2\pi)$. Using the equality $\cos(\theta_2)=2\cos^2(\tfrac{\theta_2}{2})-1$, the estimate in \eqref{est:g:1} can be continued via
\begin{equation}\label{est:g:2}
g(\theta_1,\theta_2)+g(\theta_3,\theta_2)\geq 6\cos^2\Big(\frac{\theta_2}{2}\Big)+2\cos\Big(\frac{\theta_2}{2}\Big)-2\, .
\end{equation}
Since the mapping $t\mapsto 6t^2+2t-2$ admits its minimum at $t=-1/6$, from \eqref{est:g:2} we finally deduce that
\begin{equation*}
g(\theta_1,\theta_2)+g(\theta_3,\theta_2)\geq -\frac{13}{6},
\end{equation*}
which is equivalent to \eqref{est:g}. Eventually, the case $\theta_2\in(0,\pi)$ follows similarly by exchanging the roles of $\theta_1$ and $\theta_3$ and replacing $\theta_2^\pos$ by $\theta_2^\neg$.
\end{proof}
Based on Lemma~\ref{lemma:energy of opposite chirality} we now prove Proposition~\ref{prop:compactness}.
\begin{proof}[Proof of Proposition \ref{prop:compactness}] 
We divide the proof in two steps. First, we construct a sequence $(\hat{\chi}_\e)$ of auxiliary functions $\hat{\chi}_\varepsilon : \Omega \to \{-1,1\}$ whose level sets $\{\hat{\chi}_\e=1\}$ have uniformly bounded perimeter. Second, we show that the constructed auxiliary functions are close in $L^1(\Omega)$ to the original chirality functions $\chi(u_\e)$ defined according to \eqref{def: chirality}. 

\begin{step}{1} (Compactness of the auxiliary functions) Let $\e>0$ and define $\hat{\chi}_\e : \Omega \to \{-1,1\}$ by 
\begin{equation*}
\hat{\chi}_\e \defas \begin{cases} 1 & \text{if } \chi(u_\e) > 0\, , \\
-1 &\text{otherwise.}
\end{cases}
\end{equation*}
We claim that for every $\Omega'\wcont\Omega$ we have
\begin{equation}\label{ineq:boundperimeterchi1}
\mathcal{H}^1(\partial \{\hat{\chi}_\e=1\}\cap\Omega')\leq 5 F_\e(u_\e) \,.
\end{equation}
Then the uniform bound \eqref{comp:uniformbound} together with \cite[Theorem 3.39 and Remark 3.37]{AFP} yields the existence of a function $\chi\in BV(\Omega;\{-1,1\})$ and a subsequence (not relabelled) such that $\hat{\chi}_\e \to \chi $ in $L^1(\Omega)$. To prove \eqref{ineq:boundperimeterchi1} it is convenient to consider the class of triangles
\begin{equation*}
\T_\e^\pos\defas\{T\in\T_\e(\Omega)\colon\chi(u_\e,T)>0\ \text{and}\ \chi(u_\e,T')\leq 0\ \text{for some}\ T'\in\NN_\e(T)\cap \T_\e(\Omega) \}\, ,
\end{equation*}
where $\NN_\e(T)$ is as in \eqref{def:neighbors}.
Let $\Omega'\wcont\Omega$. By the very definition of $\hat{\chi}_\e$ and of $\chi(u_\e)$ we have
\begin{equation*}
\partial\{\hat{\chi}_\e=1\}\cap\Omega'\subset\partial\big(\bigcup_{T\in\T_\e^\pos}T\big)\, ,
\end{equation*}
provided $\sqrt{3}\e<\dist(\Omega', \partial \Omega)$.  Estimating the $\H^1$-measure of the latter set in terms of the cardinality of $\T_\e^\pos$ we thus infer
\begin{equation}\label{est:perimeter}
\H^1(\partial \{\hat{\chi}_\e=1\}\cap\Omega')\leq 3\e\#\T_\e^\pos\, .
\end{equation}
The last term in \eqref{est:perimeter} can be bounded using Lemma~\ref{lemma:energy of opposite chirality}. Indeed, from Lemma \ref{lemma:energy of opposite chirality} we deduce that
\begin{equation}\label{est:cardinality}
\frac{5}{3}\e\#\T_\e^\pos\leq\sum_{T\in\T_\e(\Omega)}\sum_{T'\in\NN_\e(T)\cap\T_\e(\Omega)}F_\e(u_\e,T\cup T')\leq 3 F_\e(u_\e),
\end{equation}
where the additional factor $3$ comes from the fact that each triangle is counted $3$ times. Thus, \eqref{ineq:boundperimeterchi1} follows from \eqref{est:perimeter} and \eqref{est:cardinality}.
\end{step}

\begin{step}{2} (Closeness to $\chi(u_\e)$) We claim that for every $\delta>0$ and every $\Omega'\wcont\Omega$ there holds
\begin{equation}\label{eq:measureconvergence}
\lim_{\varepsilon\to 0}\big|\{|\hat{\chi}_\e-\chi(u_\e)| >\delta\}\cap\Omega'\big| =0 \, ,
\end{equation}
\ie the functions $\hat{\chi}_\e-\chi(u_\e)$ converge to $0$ locally in measure. Since $\|\hat{\chi}_\e-\chi(u_\e)\|_\infty \leq 2$, this implies that $(\hat{\chi}_\e-\chi(u_\e))\to 0$ in $L^1(\Omega)$, which concludes the proof of the Proposition~\ref{prop:compactness} thanks to Step 1.
It remains to prove the claim \eqref{eq:measureconvergence}. Let $\Omega'\wcont\Omega$ and $\delta>0$ and let $C_\delta$ be given by \eqref{impl:delta}. Setting
\begin{equation*}
\T_\e^\delta\defas\{T\in\T_\e(\Omega)\colon \chi(u_\e,T)\in (-1+\delta,1-\delta)\}\, ,
\end{equation*}
%
for $\e$ sufficiently small we deduce that
\begin{equation*}
|\{|\hat{\chi}_\e-\chi(u_\e)| >\delta\}\cap\Omega'|\leq\frac{\sqrt{3}}{4}\e^2\#\T_\e^\delta\leq\frac{\sqrt{3}}{4}\e C_\delta^{-1}\sum_{T\in\T_\e^\delta}F_\e(u_\e,T)\leq\frac{\sqrt{3}}{4}\e C_\delta^{-1} F_\e(u_\e)\, .
\end{equation*}
Hence,~\eqref{eq:measureconvergence} follows from the uniform bound~\eqref{comp:uniformbound}.
%
%
    \end{step}
\end{proof}
\section{Lower Bound}\label{section:lowerbound}
In this section we start proving the main result of our paper, namely Theorem \ref{maintheorem} by presenting the optimal lower bound estimate on the energy $F_{\e}$, the technically most demanding part of our contribution. We begin with a blow-up argument that gives us a first asymptotic lower bound. 
\begin{proposition}\label{prop:lowerbound}
Let $F_\e$ be as in \eqref{eq:extended energy}. Then for every $\chi\in L^1(\Omega)$ we have
\begin{align*}
\Gamma\hbox{-}\liminf_{\e\to 0}F_\e(\chi)\geq F(\chi)\, ,
\end{align*}
where $F$ is given by \eqref{def: limit energy} and the $\Gamma\hbox{-}\liminf$ is with respect to the strong topology in $L^1(\Omega)$. 
\end{proposition}
\begin{proof}


Let $\chi_\e \to \chi$ in $L^1(\Omega)$. We assume that $\liminf_{\e}F_\e(\chi_\e) < +\infty$, otherwise we have nothing to prove. Moreover, upon extracting a (not relabeled) subsequence we can assume the liminf to be a limit and hence $\sup_\e F_\e(\chi_\e)<+\infty$. In view of Remark~\ref{rem:def:extenergy} we can find a sequence of spin fields~$u_\e\in\SF_\e$ with $\chi(u_\e)=\chi_\e$ and $F_\e(\chi_\e)=F_\e(u_\e)$. In particular, $\sup_\e F_\e(u_\e)<+\infty$. Thus, from Proposition~\ref{prop:compactness} we deduce that $\chi\in BV(\Omega;\{-1,1\})$. As a consequence, to prove the statement of the proposition it suffices to show that 
\begin{equation}\label{est:liminf}
\liminf_{\e\to 0}F_\e(u_\e)\geq \int_{\js{\chi}}\varphi(\nu_\chi)\dH ,
\end{equation}
where $\varphi$ is as in \eqref{def: varphi}. To prove \eqref{est:liminf} we consider the sequence of non-negative finite Radon measures~$\mu_\e$ given by
\begin{equation*}
\mu_\e\defas\sum_{T\in\T_\e(\Omega)}\e|u_\e(\e i)+u_\e(\e j)+u_\e(\e k)|^2\delta_{\e i}\, ,
\end{equation*}
where $\delta_{\e i}$ denotes the Dirac delta in $\e i$. From the condition $\sup_\e F_\e(u_\e)<+\infty$ it follows that $\sup_\e\mu_\e(\Omega)<+\infty$, hence there exists a non-negative finite Radon measure $\mu$ such that up to subsequences (not relabeled) $\mu_\e\wsto\mu$. By the Radon-Nikod\'{y}m Theorem the measure $\mu$ can be decomposed in the sum of two mutually singular non-negative measures as 
\begin{equation*}
\mu=\mu_j\H^1\mres \js{\chi}+\mu_s\, .
\end{equation*}
Then, to establish \eqref{est:liminf} it is sufficient to show that 
\begin{equation}\label{est:muj}
\mu_j(x_0)\geq\varphi(\nu_\chi(x_0))\quad\text{for $\H^1$-a.e.}\ x_0\in\js{\chi}\, ,
\end{equation}
where $\nu_\chi(x_0)$ denotes the measure theoretic normal to $\js{\chi}$ at $x_0$. To verify \eqref{est:muj} we choose $x_0\in\js{\chi}$ satisfying
\begin{enumerate}[label=(\roman*)]
\item $\displaystyle\mu_j(x_0)=\frac{{\rm d}\mu}{{\rm d}\H^1\mres\js{\chi}}(x_0)=\lim_{\rho\to 0}\frac{\mu(Q_\rho^\nu(x_0))}{\rho}$, where we have set $\nu\defas\nu_\chi(x_0)$,
\item $\displaystyle\lim_{\rho\to 0}\frac{1}{\rho^2}\int_{Q_\rho^\nu(x_0)\cap H^\nu_+(x_0)}|\chi_\e(x)-1|\dx=0=\lim_{\rho\to 0}\frac{1}{\rho^2}\int_{Q_\rho^\nu(x_0)\cap H^\nu_-(x_0)}|\chi_\e(x)+1|\dx$,
\end{enumerate}
and we notice that (i) and (ii) are satisfied for $\H^1$-a.e.\! $x_0\in\js{\chi}$ thanks to the Besicovitch derivation Theorem and the definition of approximate jump point, respectively. Moreover, since $\mu$ is a finite Radon measure, we can choose a sequence $\rho_n\to 0$ along which $\mu(\partial Q_{\rho_n}^\nu(x_0)) = 0$. Thanks to \cite[Proposition 1.62 (a)]{AFP}, the convergence $\mu_\e\wsto\mu$ together with (i) implies that
\begin{equation}\label{est:muj1}
    \mu_j(x_0) =\lim_{n\to+\infty}\frac{\mu(Q_{\rho_{n}}^\nu(x_0))}{\rho_n} = \lim_{n\to+\infty} \lim_{\e\to 0}\frac{\mu_\e( Q_{\rho_{n}}^\nu(x_0))}{\rho_n}\geq\lim_{n\to +\infty}\hspace*{-.2em}\limsup_{\e\to 0}\frac{1}{\rho_n}F_\e(u_\e,Q_{\rho_{n}}^\nu(x_0))\, ,
    \end{equation}
where the last inequality follows from the positivity of the energy. Notice that for every $n\in\N$ there exist sequences $(\rho_n^\e)$ and $(x_0^\e)$ with $\lim_\e\rho_n^\e=\rho_n$, $\lim_\e x_0^\e=x_0$, $x_0^\e\in\L_\e$, and 
\begin{equation*}
\T_\e(Q_{\rho_n^\e}^\nu(x_0^\e))\subset\T_\e(Q_{\rho_n}^\nu(x_0))\, .
\end{equation*}
In fact, if we write $x_0$ in terms of the basis $\tb_1,\tb_2$ as $x_0=a_1\tb_1+a_2\tb_2$ for some $a_1,a_2\in\R$, we obtain the required sequence $(x_0^\e)$ by setting
\begin{equation*}
x_0^\e\defas \e\Big\lfloor\frac{a_1}{\e}\Big\rfloor\tb_1+\e\Big\lfloor\frac{a_2}{\e}\Big\rfloor\tb_2\in\L_\e\, .
\end{equation*}
Then, upon noticing that $|x_0^\e-x_0|\leq 2\e$, it suffices to set $\rho_n^\e\defas\rho_n-4\e$. Indeed, if $T\in\T_\e(Q_{\rho_n^\e}^\nu(x_0^\e))$, by definition we have that for every $x\in T$
\begin{equation*}
|\langle x-x_0^\e,\nu\rangle|<\frac{\rho_n^\e}{2}\quad\text{and}\quad|\langle x-x_0^\e,\nu^\perp\rangle|<\frac{\rho_n^\e}{2} \, ,
\end{equation*}
so that for any $x\in T$ there also holds
\begin{equation*}
|\langle x-x_0,\nu\rangle|\leq|\langle x-x_0^\e,\nu\rangle|+|x_0^\e-x_0|<\frac{\rho_n}{2}\, ,
\end{equation*}
and similarly $|\langle x-x_0,\nu^\perp\rangle|<\rho_n/2$, hence $T\in \T_\e(Q_{\rho_n}^\nu(x_0))$. As a consequence, we obtain the following estimate
\begin{equation}\label{est:shiftedenergy}
\begin{split}
\frac{1}{\rho_n}F_\e(u_\e,Q_{\rho_{n}}^\nu(x_0))&\geq\frac{\rho_n^\e}{\rho_n}\frac{\e}{\rho_n^\e}\sum_{T\in\T_\e(Q_{\rho_{n}}^\nu(x_0^\e))}|u_\e(\e i)+u_\e(\e j)+u_\e(\e k)|^2\\
&=\frac{\rho_n^\e}{\rho_n}\sum_{T\in\T_{\sigma_n^\e}(Q^\nu_{\sigma_n^\e})}\sigma_n^\e|v_{\e,n}(\sigma_n^\e i)+v_{\e,n}(\sigma_n^\e j)+v_{\e,n}(\sigma_n^\e k)|^2  ,
\end{split}
\end{equation}
where we have set $\sigma_n^\e\defas\e/\rho_n^\e$ and $v_{\e,n}(z)\defas u_\e(x_0^\e+\rho_n^\e z)$ for every $z\in\L_{\sigma_n^\e}$. Let $\chi_\nu\colon\R^2\to\{-1,1\}$ be given by
\begin{align*}
\chi_\nu(x)\defas
\begin{cases}
1 &\text{if}\ \langle x,\nu\rangle\geq 0\, ,\\
-1 &\text{if}\ \langle x,\nu\rangle < 0\, .
\end{cases}
\end{align*}
Then (ii) ensures that $\chi(v_{\e,n})\to\chi_\nu$ in $L^1(Q^\nu)$ as first $\e\to 0$ and then $n\to+\infty$. Thus, gathering \eqref{est:muj1}--\eqref{est:shiftedenergy} and applying a diagonal argument we find a sequence $\sigma_m\defas\e_m/\rho_{n_m}$ converging to $0$ as $m\to+\infty$ such that for $v_m\defas v_{\e_m,n_m}$ there holds $\chi(v_m)\to \chi_\nu$ in $L^1(Q^\nu)$ and
\begin{equation*}
\mu_j(x_0)\geq\liminf_{m\to+\infty}F_{\sigma_m}(v_m,Q^\nu)\, .
\end{equation*}
For $\ell,h>0$ let us finally introduce the minimization problem
\begin{equation}\label{def:psi}
\psi(\ell,h,\nu)\defas\frac{1}{\ell}\inf\big\{\liminf_{\e\to 0}F_\e(u_\e,R_{\ell,h}^\nu)\colon\chi(u_\e)\to\chi_\nu\ \text{in}\ L^1(R_{\ell,h}^\nu)\big\}\, ,
\end{equation}
so that the sequence $(v_m)$ is admissible for $\psi(1,1,\nu)$. Then~\eqref{est:muj} follows from Proposition \ref{prop:psi is phi} below, concluding the proof of Proposition~\ref{prop:lowerbound}.
\end{proof}
\begin{proposition}\label{prop:psi is phi}
Let $\psi$ be the function defined in~\eqref{def:psi}.  Then $\psi(1,1,\nu)\geq\varphi(\nu)$ for every $\nu\in\S^1$. 
\end{proposition}
To prove Proposition \ref{prop:psi is phi} it is necessary to modify admissible sequences for the infimum problem defining $\psi(1,1,\nu)$ in such a way that they satisfy the boundary conditions required in the minimum problem defining $\varphi(\nu)$, without essentially increasing the energy. This will be done by a careful interpolation procedure based on several auxiliary results and estimates that we prefer to state in separate lemmas below. As a first step towards the proof of Proposition \ref{prop:psi is phi} we show that $\psi(\ell,h,\nu)$ is independent of $\ell$ and $h$, which in turn will allow us to conclude that the energy of admissible functions for $\psi(1,1,\nu)$ concentrates close to the line segment $L^\nu$ (see Lemma \ref{lemma:choosing strip} below).
\begin{lemma}\label{lem:independence}
Let $\psi \colon (0,+\infty)\x (0,+\infty)\x\S^1\to[0,+\infty]$ be given by \eqref{def:psi}; then $\psi(\cdot,\cdot,\nu)$ is independent of $\ell,h$ for every $\nu\in\S^1$.
\end{lemma}
\begin{proof}
Let $\nu\in\S^1$ be fixed. To show that $\psi(\cdot,\cdot,\nu)$ does not depend on $\ell,h$ it suffices to show that for every $\ell,h,\lambda>0$ the following identities hold
\begin{align}\label{ind:eq:1}
\psi(\lambda\ell,h,\nu)=\psi(\ell,h,\nu)\qquad\text{and}\qquad\psi(\ell,\lambda h,\nu)=\psi(\ell,h,\nu)\, .
\end{align}
Let us fix $\ell,h>0$. We first observe that
\begin{align}
\psi(\ell,\lambda h,\nu) &\geq\psi(\ell,h,\nu)\quad\text{for every}\ \lambda\in[1,+\infty) \, ,\label{ind:est:1a}\\
\psi(\ell,h,\nu) &\geq\psi(\ell,\lambda h,\nu)\quad\text{for every}\ \lambda\in (0,1) \, ,\label{ind:est:1b}
\end{align}
since $F_\e$ is increasing as a set function. The proof of \eqref{ind:eq:1} is now divided into three steps.

\begin{step}{1}
$\psi$ is invariant under dilations, \ie 
\begin{align}\label{ind:eq:2}
\psi(\lambda \ell,\lambda h,\nu)=\psi(\ell,h,\nu)\quad\text{for every}\ \lambda>0 \, .
\end{align}
Let $(u_\e)$ be any sequence of spin fields $u_\e\colon\L_\e\to\S^1$ with $\chi(u_\e)\to\chi_{\nu}$ in $L^1(R_{\lambda \ell,\lambda h}^\nu)$. We define the rescaled functions $v_\e:\L_{\e/\lambda}\to\S^1$ by setting $v_\e(z)\defas u_\e(\lambda z)$ for every $z\in\L_{\e/\lambda}$. Then $\chi(v_\e)\to \chi_\nu$ in $L^1(R_{\ell,h}^{\nu})$ and
\begin{equation*}
\begin{split}
F_{\frac{\e}{\lambda}}(v_\e,R_{\ell,h}^\nu) &=\sum_{T\in\T_{\frac{\e}{\lambda}}(R_{\ell,h}^\nu)}\frac{\e}{\lambda}|v_\e(\tfrac{\e}{\lambda} i)+v_\e(\tfrac{\e}{\lambda} j)+v_\e(\tfrac{\e}{\lambda} k)|^2\\
&=\frac{1}{\lambda}\sum_{T\in\T_\e(R_{\lambda \ell,\lambda h}^\nu)}\e|u_\e(\e i)+u_\e(\e j)+u_\e(\e k)|^2
=\frac{1}{\lambda}F_\e(u_\e,R_{\lambda \ell,\lambda h}^\nu) \, .
\end{split}
\end{equation*}
Setting $\eta\defas\e/\lambda\to 0$ as $\e\to 0$ and passing to the infimum over all admissible sequences $(u_\e)$ we deduce that
\begin{align*}
\psi(\lambda \ell,\lambda h,\nu)&\geq\frac{1}{\ell}\inf\big\{\liminf_{\eta\to 0}F_\eta(v_\eta,R_{\ell,h}^\nu)\colon\chi(v_\eta)\to\chi_{\nu}\ \text{in}\ L^1(R_{\ell,h}^\nu)\big\}=\psi(\ell,h,\nu) \, .
\end{align*}
The opposite inequality and hence \eqref{ind:eq:2} follow by observing that
\begin{align*}
\psi(\ell,h,\nu)=\psi(\lambda^{-1}(\lambda \ell),\lambda^{-1}(\lambda h),\nu)\geq\psi(\lambda \ell,\lambda h,\nu) \, .
\end{align*}
Note that thanks to \eqref{ind:eq:2} it suffices to show the first equality in \eqref{ind:eq:1}. In fact, if the first equality in \eqref{ind:eq:1} is true, from \eqref{ind:eq:2} we directly deduce that
\begin{equation*}
\psi(\ell,\lambda h,\nu)=\psi(\lambda^{-1}\ell, h,\nu)=\psi(\ell,h,\nu)\quad\text{for every}\ \lambda>0\, .
\end{equation*} 
\end{step}
\noindent\begin{step}{2}
We continue establishing the first equality in \eqref{ind:eq:1} by showing that
\begin{align}\label{ind:eq:3}
\psi(N\ell,h,\nu)=\psi(\ell,h,\nu)\quad\text{for every}\ N\in\N \, .
\end{align}
For $N\in\N$ fixed let $(u_\e)$ be a sequence of spin fields satisfying $\chi(u_\e)\to\chi_\nu$ in $L^1(R_{N\ell,h}^\nu)$. We subdivide the rectangle $R_{N\ell,h}^\nu$ in $N$ open rectangles of the form
\begin{align*}
R_{\ell,h}^\nu(x_m)\quad\text{with}\ x_m\defas\Big( m - \frac{N-1}{2}  \Big) \ell\nu^\perp\ \text{for}\ m\in\{0,\dots,N-1\}\, .
\end{align*}
Notice that $x\in R_{\ell,h}^\nu(x_m)$ if and only if
\begin{equation*}
\big|\langle x,\nu^\perp\rangle-\Big(m-\frac{N-1}{2}\Big)\ell\big|<\frac{\ell}{2}\quad\text{and}\quad |\langle x,\nu\rangle|<\frac{h}{2} \, ,
\end{equation*}
and therefore $R_{\ell,h}^\nu(x_m)\subset R_{N\ell,h}^\nu$ for all $m \in \{0,\ldots,N-1\}$. 
By choosing $m_0\in\{0,\ldots,N-1\}$ such that $F_\e(u_\e,R_{\ell,h}^\nu(x_{m_0}))\leq F_\e(u_\e,R_{\ell,h}^\nu(x_m))$ for every $m\in\{0,\ldots,N-1\}$ we obtain the estimate
\begin{align}\label{ind:est:3}
\frac{1}{N\ell}F_\e(u_\e,R_{N\ell,h}^\nu)\geq\frac{1}{N\ell}\sum_{m=0}^{N-1}F_\e(u_\e,R_{\ell,h}^\nu(x_m))\geq \frac{1}{\ell} F_\e(u_\e,R_{\ell,h}^\nu(x_{m_0})) \, .
\end{align}
We now define a suitable shifted version of $u_\e$ whose energy is concentrated in a rectangle centered at zero. To this end, as in the proof of Proposition \ref{prop:lowerbound} it is convenient to write the vector $\nu^\perp$ in terms of the basis $\{\tb_1,\tb_2\}$ as $\nu^\perp=a_1\tb_1+a_2\tb_2$ for some $a_1,a_2\in\R$ and to introduce the vector $x_{m_0}^\e\in\L_\e$ given by
\begin{align*}
x_{m_0}^\e\defas\e\Big\lfloor\frac{\big(m_0-\frac{N-1}{2}\big)\ell a_1}{\e}\Big\rfloor\tb_1+\e\Big\lfloor\frac{\big(m_0-\frac{N-1}{2}\big)\ell a_2}{\e}\Big\rfloor \tb_2\, .
\end{align*}
We then define spin fields $v_\e\colon\L_\e\to\S^1$ by setting $v_\e(z)\defas u_\e(z+x_{m_0}^\e)$.
As in the proof of Proposition~\ref{prop:lowerbound} we notice that $|x_{m_0}^\e-x_{m_0}|\leq 2\e$, $\chi(v_\e)\to \chi_\nu$ in $L^1(R_{\ell,h}^\nu)$ and $R_{\ell-4\e,h-4\e}^\nu(x_{m_0}^\e)\subset R_{\ell,h}^\nu(x_{m_0})$. Let us fix $\lambda\in (0,1)$ and $\e_\lambda>0$ sufficiently small such that $\ell-4\e_\lambda>\lambda\ell$, $h-4\e_\lambda>\lambda h$. Then for every $\e\in (0,\e_\lambda)$ there holds $\T_\e(R_{\lambda\ell,\lambda h}^\nu(x_{m_0}^\e))\subset \T_\e(R_{\ell,h}^\nu(x_{m_0}))$, hence
\begin{equation*}
\frac{1}{\ell} F_\e(v_\e,R_{\lambda\ell,\lambda h}^\nu)\leq \frac{1}{\ell}F_\e(u_\e,R_{\ell,h}^\nu(x_{m_0}))\, .
\end{equation*}
Moreover, since $v_\e$ is admissible for $\psi(\lambda\ell,\lambda h,\nu)$, we have
\begin{equation}\label{ind:est:5}
\lambda\psi(\lambda \ell,\lambda h,\nu)\leq\frac{1}{\ell}\liminf_{\e\to 0}F_\e(v_\e,R_{\lambda \ell,\lambda h}^\nu)\, .
\end{equation}
Combining \eqref{ind:eq:2} in Step 1 with \eqref{ind:est:3}--\eqref{ind:est:5}, in view of the arbitrariness of $u_\e$ we finally obtain
\begin{equation*}
\lambda\psi(\ell,h,\nu)=\lambda\psi(\lambda\ell,\lambda h,\nu)\leq \psi(N\ell,h,\nu)\, .
\end{equation*}
Thus, by letting $\lambda\to 1$ we deduce that $\psi(\ell,h,\nu)\leq\psi(N\ell,h,\nu)$. Finally, \eqref{ind:eq:3} follows from \eqref{ind:eq:2} and \eqref{ind:est:1b} by observing that
\begin{equation*}
\psi(\ell,h,\nu)\leq\psi(N\ell,h,\nu)=\psi(\ell,\tfrac{h}{N},\nu)\leq\psi(\ell,h,\nu)\, .
\end{equation*}
\end{step}

\begin{step}{3}
We prove the first equality in \eqref{ind:eq:1}. Suppose first that $\lambda\in(0,+\infty)\cap\Q$. Then $\lambda=N/M$ for some $N,M\in\N$, hence applying twice \eqref{ind:eq:3} yields
\begin{equation}\label{ind:eq:4}
\psi(\lambda\ell,h,\nu)=\psi(\tfrac{N}{M}\ell,h,\nu)=\psi(\tfrac{1}{M}\ell,h,\nu)=\psi(M(\tfrac{1}{M}\ell),h,\nu)=\psi(\ell,h,\nu)\, .
\end{equation}
Suppose now that $\lambda\in(0,+\infty)$ and let $(\lambda_n)\subset(0,+\infty)\cap\Q$ with $\lambda_n\to\lambda$ as $n\to+\infty$, $\lambda_n>\lambda$ for every $n\in\N$. Thanks to \eqref{ind:eq:2} and \eqref{ind:eq:4} we deduce that
\begin{equation*}
\psi(\lambda\ell,h,\nu)=\psi(\lambda_n\ell,\tfrac{\lambda_n}{\lambda}h,\nu)=\psi(\ell,\tfrac{\lambda_n}{\lambda}h,\nu)\geq\psi(\ell,h,\nu)\, ,
\end{equation*}
where the last inequality follows from \eqref{ind:est:1a}, since $\lambda_n/\lambda>1$. To prove the opposite inequality it suffices to take a sequence $(\lambda_n)\subset(0,+\infty)\cap\Q$ converging to $\lambda$ with $\lambda_n<\lambda$. Then, arguing as before and now applying \eqref{ind:est:1a} we obtain
\begin{equation*}
\psi(\lambda\ell,h,\nu)=\psi(\ell,\tfrac{\lambda_n}{\lambda}h,\nu)\leq\psi(\ell,h,\nu)\, ,
\end{equation*} 
hence equality follows.
\end{step}
\end{proof}
On account of Lemma~\ref{lem:independence} we show that for a sequence $(u_\e)$ realizing the infimum in the definition of $\psi(1,1,\nu)$ the energy concentrates close to the line $L^\nu$. As a consequence, we obtain that outside a small neighborhood of $L^\nu$ there exists a suitable strip on which the energy is of order $o(\e)$. 
To be more precise, for fixed $\nu\in\S^1$, $\delta>0$, and every $\e>0$ we introduce the class $\mathscr{S}_{\e,\delta}^\nu$ of strips
\begin{equation}\label{def:strips}
\mathscr{S}_{\e,\delta}^\nu\defas\Big\{Q_{r+\cnu\e}^\nu\setminus \big(\overline{Q}_{r}^\nu\cup\overline{R}_{1,\delta}^\nu\big)\colon r\in(1-3\delta,1-2\delta)\Big\}\, .
\end{equation}
We denote the elements of $\mathscr{S}_{\e,\delta}^\nu$ by $S_{\e,r}$. Then the following result holds true.
\begin{lemma}\label{lemma:choosing strip}
Let $\nu\in\S^1$ and let $(u_\e)$ be a sequence such that $\chi(u_\e)\to\chi_\nu$ in $L^1(Q^\nu)$ and $F_\e(u_\e,Q^\nu)\to\psi(1,1,\nu)$. Then for every $\delta>0$ there exists a sequence $\sigma_\e\to0$ (depending on $\delta$) and a strip $S_\e=S_{\e,r_\e}\in\mathscr{S}_{\e,\delta}^\nu$ such that
\begin{equation}\label{est:FeSe}
F_\e(u_\e,S_\e)+\|\chi(u_\e)-\chi_\nu\|_{L^1(S_\e)}\leq \e\sigma_\e\, .
\end{equation}
\end{lemma}
\begin{proof}
Let $\nu\in\S^1$ and $(u_\e)$ be as in the statement and let $\delta>0$ be fixed. For every Borel set $A\subset Q^\nu$ set
\begin{equation*}
G_\e(u_\e, A)\defas F_\e(u_\e,A)+\int_A|\chi(u_\e)-\chi_\nu|\dx\, .
\end{equation*}
We consider for $\e$ small enough the family of pairwise disjoint strips $S_{\e,r_\e^m}\in\mathscr{S}_{\e,\delta}^\nu$ with $r_\e^m=1-3\delta+\cnu m\e$ and $m\in\{0,\ldots,\lfloor\tfrac{\delta}{\cnu\e}\rfloor-1\}$ and we notice that
\begin{equation*} 
\bigcup_{m=0}^{\lfloor\frac{\delta}{\cnu\e}\rfloor-1}S_{\e,r_\e^m}\subset Q_{1-2\delta}^\nu\setminus(\overline{Q}_{1-3\delta}^\nu\cup \overline{R}_{1,\delta}^\nu)\subset Q^\nu\setminus\overline{R}_{1,\delta}^\nu\, .
\end{equation*}
This implies in particular that
\begin{equation*}
\sum_{m=0}^{\lfloor\frac{\delta}{\cnu\e}\rfloor-1}G_\e(u_\e,S_{\e,r_\e^m})\leq G_\e\Big(u_\e,\bigcup_{m=0}^{\lfloor\frac{\delta}{\cnu\e}\rfloor-1}S_{\e,r_\e^m}\Big)\leq F_\e(u_\e,Q^\nu\setminus\overline{R}_{1,\delta}^\nu)+\int_{Q^\nu\setminus\overline{R}_{1,\delta}^\nu}|\chi(u_\e)-\chi_\nu|\dx\, .
\end{equation*}
Averaging over $m\in\{0,\dots,\lfloor\tfrac{\delta}{\cnu\e}\rfloor-1\}$ we thus find $m(\e)$ such that the strip $S_{\e,r_\e^{m(\e)}}$ satisfies
\begin{equation}\label{est:Ge:averaging}
G_\e(u_\e,S_{\e,r_\e^{m(\e)}})\leq\Big\lfloor\frac{\delta}{\cnu\e}\Big\rfloor^{-1}\big(F_\e(u_\e,Q^\nu\setminus\overline{R}_{1,\delta}^\nu)+\|\chi(u_\e)-\chi_\nu\|_{L^1(Q^\nu)}\big)\, .
\end{equation}
Notice that $F_\e(u_\e,Q^\nu\setminus\overline{R}_{1,\delta}^\nu)\to 0$ as $\e\to 0$. In fact, Lemma \ref{lem:independence} together with the choice of $(u_\e)$ yields
\begin{align*}
\psi(1,1,\nu) &=\lim_{\e\to 0}F_\e(u_\e,Q^\nu)\geq\limsup_{\e\to 0}F_\e(u_\e,R_{1,\delta}^\nu)\geq\liminf_{\e\to 0}F_\e(u_\e,R_{1,\delta}^\nu)\geq\psi(1,\delta,\nu)=\psi(1,1,\nu)\, ,
\end{align*}
from which we readily deduce that $F_\e(u_\e,R_{1,\delta}^\nu)\to\psi(1,1,\nu)$ as $\e\to 0$, hence
\begin{equation*}
F_\e(u_\e,Q^\nu\setminus\overline{R}_{1,\delta}^\nu)\leq F_\e(u_\e,Q^\nu)-F_\e(u_\e,R_{1,\delta}^\nu)\ \to 0\ \text{as}\ \e\to 0\, .
\end{equation*}
Thus, in view of~\eqref{est:Ge:averaging}, it suffices to set $\sigma_\e\defas\tfrac{13}{\delta}(F_\e(u_\e,Q^\nu\setminus\overline{R}_{1,\delta}^\nu)+\|\chi(u_\e)-\chi_\nu\|_{L^1(Q^\nu)})$ and $r_\e\defas r_\e^{m(\e)}$ to find the required strip $S_{\e, r_\e}\in\mathscr{S}_{\e,\delta}^\nu$ satisfying \eqref{est:FeSe}.
\end{proof}
We are now in a position to start with the interpolation procedure mentioned before. The final interpolation procedure will be based on a one-dimensional construction that we introduce below. 

\noindent\textbf{One-dimensional interpolation}. To define the one-dimensional interpolation we consider slices in the triangular lattice. To this end, let $\tb_1$, $\tb_2$, and $\tb_3$ be as in Section~\ref{subsec:lattice}. Given $\alpha \in \{1,2,3\}$ we consider the orthogonal vector $\tb_\alpha^\perp$ to $\tb_\alpha$ and we define the slice in the direction~$\tb_\alpha$ by
        \begin{equation*} 
            \Sigma^{\alpha}  := \big\{ s \tb_\alpha +  t \tb_\alpha^\perp \colon s \in \R \, , \ t \in [0, \tfrac{\sqrt{3}}{2}]  \big\} \, .
        \end{equation*}
        Given $z \in \Z$, we define
        \begin{equation*}
            \Sigma^{\alpha,z} := \Sigma^{\alpha} + \tfrac{\sqrt{3}}{2} z \, \tb_\alpha^\perp = \big\{ s \tb_\alpha +  t \tb_\alpha^\perp \colon s \in \R \, , \ t \in [\tfrac{\sqrt{3}}{2} z, \tfrac{\sqrt{3}}{2}(z+1)]  \big\} \, .
        \end{equation*}
        Finally, for every $\e$ we set 
        \begin{equation} \label{def:slice}
            \Sigma^{\alpha,z}_\e := \e \Sigma^{\alpha,z} .
        \end{equation}

        We shall define the one-dimensional interpolation in a slice $\Sigma^{\alpha}$ starting from a triangle $T_0\in\T(\R^2)$ such that $T_0 \subset \Sigma^{\alpha}$. Let us denote by $i_0  \in \L^1$, $j_0   \in \L^2$, $k_0   \in\L^3$ the vertices of $T_0$. Note that $\langle i_0, \tb_\alpha^\perp \rangle, \langle j_0, \tb_\alpha^\perp \rangle, \langle k_0, \tb_\alpha^\perp \rangle \in \{0,\tfrac{\sqrt{3}}{2}\}$. We define the lattice points $i_h \in \L^1$, $j_h \in \L^2$, $k_h\in\L^3$ and the triangle $T_h$ with the following recursive formula: we set $\tau(0) := 1$, $\tau(\tfrac{\sqrt{3}}{2}) := -1$ and for $h\in\N$
\begin{equation} \label{def:vertices in slice}
    \begin{split}
        i_{h+1} &:= i_h +   \tb_\alpha + \tfrac{1}{2} \tb_\alpha + \tau(\langle i_h , \tb_\alpha^\perp \rangle) \tfrac{\sqrt{3}}{2} \tb_\alpha^\perp \, ,\\
        j_{h+1} &:= j_h + \tb_\alpha + \tfrac{1}{2} \tb_\alpha + \tau(\langle j_h , \tb_\alpha^\perp \rangle) \tfrac{\sqrt{3}}{2} \tb_\alpha^\perp \, ,\\
        k_{h+1} &:= k_h + \tb_\alpha + \tfrac{1}{2} \tb_\alpha + \tau(\langle k_h , \tb_\alpha^\perp \rangle) \tfrac{\sqrt{3}}{2} \tb_\alpha^\perp \, , \\
        T_{h+1} & := \conv\{i_{h+1},j_{h+1},k_{h+1}\} \subset \Sigma^\alpha \, ,
    \end{split}
\end{equation}
(see Figure \ref{fig:interpolation}).
Observe that $\tau(\langle i_{h+1} , \tb_\alpha^\perp \rangle) = -\tau(\langle i_h , \tb_\alpha^\perp \rangle)$, the analogous equality being true also for~$j_h$ and~$k_h$. Moreover, $T_{2h} = T_0 + 3 h \tb_\alpha$.

We define the half-slice $\Sigma^\alpha(T_0)$ of the lattice $\L$ starting from $T_0$ by
\begin{equation}\label{def:lattice slice}
\Sigma^\alpha(T_0)\defas\conv\{T_h \colon h \in \N \}\, .
\end{equation} 

Given $u\colon\L\to\S^1$ and $N,m\in\N$, we now define in the half-slice $\Sigma^\alpha(T_0)$ a one-parameter family (parametrized by $m$) of spin fields which coincides with $u$ on $T_0$ and with the fixed ground state $u^\pos$ on $T_h$ for $h \geq N$. We construct the interpolation in such a way that the configuration of spins rotates a fixed amount of times by $2\pi$. To make the construction precise, we first say that the three angles $\theta(i_0)\in \R$ (not necessarily in $[0,2\pi)$), $\theta(j_0)\in[\theta(i_0)-\pi,\theta(i_0)+\pi)$ and $\theta(k_0)\in[\theta(j_0)-\pi,\theta(j_0)+\pi)$ represent a lifting of $u$ in $T_{0}$ if $u(i_0)=\exp(\iota \theta(i_0))$, $u(j_0)=\exp(\iota \theta(j_0))$ and $u(k_0)=\exp(\iota \theta(k_0))$. We then define the interpolated angles $\theta(i_h),\theta(j_h),\theta(k_h)$ for $h=0,\dots,N$ by
\begin{equation} \label{def:interpolated angles}
    \begin{split}
        \theta(i_h) & \defas\theta(i_0)+h\frac{2\pi m-\theta(i_0)}{N} = \Big(1 - \frac{h}{N}\Big)\theta(i_0) + \frac{h}{N}2\pi m  \, , \\
        \theta(j_h) & \defas\theta(j_0)+h\frac{2\pi m+\frac{2\pi}{3}-\theta(j_0)}{N} = \Big(1 - \frac{h}{N}\Big)\theta(j_0) + \frac{h}{N}2\pi m + \frac{h}{N}\frac{2\pi}{3} \, , \\
        \theta(k_h) & \defas\theta(k_0)+h\frac{2\pi m+\frac{4\pi}{3}-\theta(k_0)}{N} =  \Big(1 - \frac{h}{N}\Big)\theta(k_0) + \frac{h}{N}2\pi m + \frac{h}{N}\frac{4\pi}{3}\, , \\
    \end{split}
\end{equation}
and $\theta(i_h) := 2\pi m$, $\theta(j_h) := 2\pi m+\frac{2\pi}{3}$, $\theta(k_h) := 2\pi m+\frac{4\pi}{3}$ for $h \geq N+1$ (see Figure \ref{fig:interpolation}).
Eventually, we define $u^{N,m}\colon\L\cap\Sigma^\alpha(T_0)\to\S^1$ by setting
\begin{equation}\label{def:interpolation}
    u^{N,m}(i_h)\defas\exp(\iota \theta(i_h))\, ,\quad u^{N,m}(j_h)\defas\exp(\iota \theta(j_h))\, ,\quad u^{N,m}(k_h)\defas\exp(\iota \theta(k_h))\, .
\end{equation}
Note that  $u^{N,m} = u^\pos$ on $T_h$ for $h \geq N$.

\begin{figure}[H]
   \includegraphics{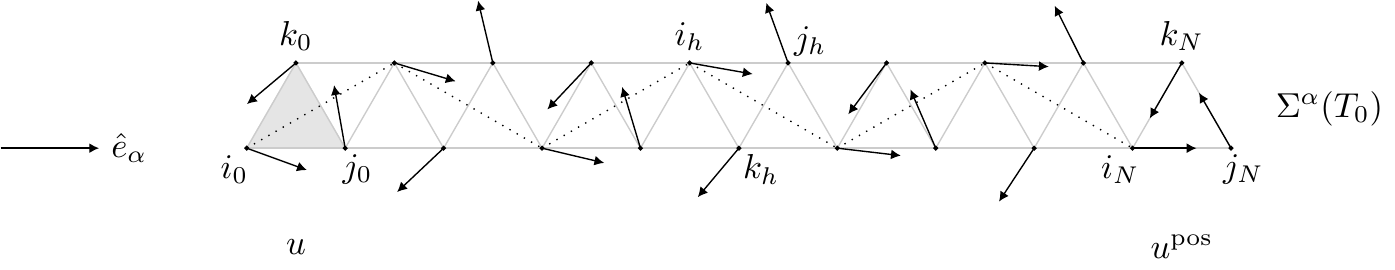}
\caption{Example of interpolation from $u$ to $u^\pos$ in the slice $\Sigma^\alpha(T_0)$ starting from the triangle $T_0$ (in grey).}\label{fig:interpolation}
\end{figure}

In the next lemma we estimate the energy of the interpolation on $\Sigma^\alpha(T_0)$ in terms of the energy on the initial triangle $T_0$ plus an error depending on the number of steps $N$ and on $m$. We assume that the configuration of spins in the initial triangle is sufficiently close to a ground state with chirality 1 (not necessarily $u^\pos$).

\begin{lemma} \label{lemma:1d interpolation}
Let $T_0\in\T(\R^2)$ be a triangle of vertices $i_{0}\in\L^{1}$, $j_{0}\in\L^{2}$, and $k_{0}\in\L^{3}$. Let $u\colon\L\to\S^1$ and let $\theta(i_0) \in \R$, $\theta(j_0) \in [\theta(i_0) - \pi, \theta(i_0) + \pi)$ and $\theta(k_0) \in [\theta(j_0)-\pi, \theta(j_0)+\pi)$ be three angles representing a lifting of $u$ in $T_{0}$ satisfying 
    \begin{equation} \label{assumption:angles}
         \Big|  \theta(j_0) - \theta(i_0) -\frac{2\pi}{3} \Big|  \leq \frac{1}{4}  \, , \quad  \Big|  \theta(k_0) - \theta(j_0) -\frac{2\pi}{3} \Big|  \leq \frac{1}{4}  \, .
    \end{equation}
    Let $N,m \in \N$ and assume that 
    \begin{equation} \label{assumption:m}
        2\pi m \geq |\theta(i_0)|+2\pi \, .
    \end{equation}
    Let $u^{N,m}$ be the interpolation on $\Sigma^\alpha(T_0)$ defined according to~\eqref{def:interpolation}. Then there exists a constant $C>0$ independent of $N$ and $m$ such that
    \begin{equation*} 
    F_1(u^{N,m},\Sigma^\alpha(T_0))\leq C \Big(N F_1(u,T_0)+\frac{m^2}{N}\Big) \, .
    \end{equation*}
    \end{lemma}

    %
    \begin{proof} It is not restrictive to assume that $j_0-i_0=\tb_\alpha$ as in Figure \ref{fig:interpolation}. We shall estimate each of the terms in the sum
        \begin{equation} \label{eq:energy in slice}
            \begin{split}
                 F_1(u^{N,m},\Sigma^\alpha(T_0)) = & \sum_{h=0}^{N-1} |u^{N,m}(i_h) + u^{N,m}(j_h) + u^{N,m}(k_h)|^2  \\
                 + & \sum_{h=0}^{N-1} |u^{N,m}(i_{h+1}) + u^{N,m}(j_h) + u^{N,m}(k_h)|^2 \\
                  + & \sum_{h=0}^{N-1} |u^{N,m}(i_{h+1}) + u^{N,m}(j_h) + u^{N,m}(k_{h+1})|^2,
            \end{split}
        \end{equation}
        where we used that for $h \geq N$ we have that 
        \begin{equation*}
        |u^{N,m}(i_h) + u^{N,m}(j_h) + u^{N,m}(k_h)|^2 = |u^\pos(i_h) + u^\pos(j_h) + u^\pos(k_h)|^2 = 0\, ,
        \end{equation*}
         being $u^\pos$ a ground state. Adopting the notation for the angles used in the construction in~\eqref{def:interpolation}, we recast the energy in the first term of the sum as
        \begin{equation} \label{eq:energy with cos}
            \begin{split}
                & |u^{N,m}(i_h) + u^{N,m}(j_h) + u^{N,m}(k_h)|^2 \\
                & \quad = 3 + 2 \cos(\theta(j_h) - \theta(i_h)) + 2 \cos(\theta(k_h) - \theta(j_h)) + 2 \cos(\theta(i_h) - \theta(k_h)) \, .
            \end{split}
        \end{equation}
        Note that, by~\eqref{def:interpolated angles} and~\eqref{assumption:angles},
        \begin{equation} \label{est:angles}
            \begin{split}
                 \Big| \theta(j_h) - \theta(i_h) - \frac{2\pi}{3} \Big| & \leq \Big|  \theta(j_0) - \theta(i_0) -\frac{2\pi}{3} \Big|  \leq \frac{1}{4}  \, , \\
                 \Big| \theta(k_h) - \theta(j_h) - \frac{2\pi}{3} \Big| & \leq \Big|  \theta(k_0) - \theta(j_0) -\frac{2\pi}{3} \Big|  \leq \frac{1}{4}  \, , \\
                 \Big| \theta(k_h) - \theta(i_h)  - \frac{4\pi}{3} \Big|  & \leq \Big| \theta(k_0) - \theta(i_0)   - \frac{4\pi}{3} \Big|  \leq \frac{1}{2} \, .
            \end{split}
        \end{equation}
By Taylor's formula, there exists $\zeta\in[\phi,2\pi/3]$ such that $1 + 2\cos(\phi) = - \sqrt{3} (\phi-\frac{2\pi}{3}) + \frac{1}{2} (\phi-\frac{2\pi}{3})^2 + \frac{1}{3} \sin(\zeta) (\phi-\frac{2\pi}{3})^3$. As a result we obtain the estimates
        \begin{equation*} 
             \frac{1}{3} \Big(\phi-\frac{2\pi}{3}\Big)^{\! 2} \leq 1 + 2\cos(\phi) + \sqrt{3}  \Big(\phi-\frac{2\pi}{3}\Big) \leq  \frac{2}{3} \Big(\phi-\frac{2\pi}{3}\Big)^{\! 2} , \quad \text{for } \Big|\phi-\frac{2\pi}{3}\Big| \leq \frac{1}{2} \, .
        \end{equation*} 
        Analogously,
        \begin{equation*} 
              \frac{1}{3} \Big(\phi-\frac{4\pi}{3}\Big)^{\! 2} \leq 1 + 2\cos(\phi) - \sqrt{3}  \Big(\phi-\frac{4\pi}{3}\Big) \leq   \frac{2}{3} \Big(\phi-\frac{4\pi}{3}\Big)^{\! 2} , \quad   \text{for } \Big|\phi-\frac{4\pi}{3}\Big| \leq \frac{1}{2} \, .
        \end{equation*}
        Then by~\eqref{eq:energy with cos},~\eqref{est:angles}, and the two previous estimates we infer that  
        \begin{equation*}
            \begin{split}
                & |u^{N,m}(i_h) + u^{N,m}(j_h) + u^{N,m}(k_h)|^2 \\
                 & \quad \leq \frac{2}{3} \Big[ \Big(\theta(j_h) - \theta(i_h)-\frac{2\pi}{3}\Big)^{\! 2} +   \Big(\theta(k_h) - \theta(j_h)-\frac{2\pi}{3}\Big)^{\! 2} +   \Big(\theta(k_h) - \theta(i_h) - \frac{4\pi}{3}\Big)^{\! 2} \Big]\\
                & \quad \leq  \frac{2}{3} \Big[ \Big(\theta(j_0) - \theta(i_0)-\frac{2\pi}{3}\Big)^{\! 2} +   \Big(\theta(k_0) - \theta(j_0)-\frac{2\pi}{3}\Big)^{\! 2} + \Big(\theta(k_0) - \theta(i_0)  - \frac{4\pi}{3}\Big)^{\! 2} \Big] \\
                & \quad \leq \frac{2}{3} 3 |u^{N,m}(i_0) + u^{N,m}(j_0) + u^{N,m}(k_0)|^2 = 2  F_1(u,T_0) \, .
            \end{split}
            \end{equation*}
            This proves that 
            \begin{equation*} 
                \sum_{h=0}^{N-1} |u^{N,m}(i_h) + u^{N,m}(j_h) + u^{N,m}(k_h)|^2 \leq 2 N F_1(u,T_0)  \, .
            \end{equation*}
    
            Let us now consider the second term in the sum in the right-hand side of~\eqref{eq:energy in slice}. For every $h=0,\dots,N-1$ we have
            \begin{equation*}
                \begin{split}
                    & |u^{N,m}(i_{h+1}) + u^{N,m}(j_h) + u^{N,m}(k_h)|^2 \\
                    & \quad \leq 2 \, |u^{N,m}(i_{h}) + u^{N,m}(j_h) + u^{N,m}(k_h)|^2 + 2 \, |u^{N,m}(i_{h+1}) - u^{N,m}(i_h)|^2 .
                \end{split}
            \end{equation*}
            The first term is estimated as via $2F_1(u,T_0)$. As for $|u^{N,m}(i_{h+1}) - u^{N,m}(i_h)|^2$, by~\eqref{def:interpolated angles} we have that 
            \begin{equation*}
                    |u^{N,m}(i_{h+1}) - u^{N,m}(i_h)|^2 = 2 - 2 \cos(\theta(i_{h+1}) - \theta(i_h)) = 2 - 2 \cos\Big(\frac{2\pi m-\theta(i_0)}{N}\Big)\, .  
			\end{equation*}
			Using the fact that $1-\cos(t)\leq \tfrac{t^2}{2}$ we deduce
			\begin{equation*}
                    |u^{N,m}(i_{h+1}) - u^{N,m}(i_h)|^2 \leq \Big(\frac{2\pi m-\theta(i_0)}{N}\Big)^{\! 2} \leq C \frac{m^2}{N^2} \, , 
            \end{equation*}
            since $|\theta(i_0)| \leq  2 \pi m$.   Hence
            \begin{equation*}
                \sum_{h=0}^{N-1} |u^{N,m}(i_{h+1}) + u^{N,m}(j_h) + u^{N,m}(k_h)|^2 \leq C N F_1(u,T_0) + C \frac{ m^2}{N} \, .
            \end{equation*}
            
            The third term in the right-hand side in~\eqref{eq:energy in slice} is treated analogously using the inequality
            \begin{equation*}
                \begin{split}
                    |u^{N,m}(k_{h+1}) - u^{N,m}(k_h)|^2 & = 2 - 2 \cos(\theta(k_{h+1}) - \theta(k_h)) = 2 - 2 \cos\Big(\frac{2\pi m+\frac{4\pi}{3}-\theta(k_0)}{N}\Big)  \\
                    & \leq \Big(\frac{2\pi m+\frac{4\pi}{3}-\theta(k_0)}{N}\Big)^{\! 2}  \leq C \frac{ m^2}{N^{2}}\, ,
                \end{split}
            \end{equation*}
          where we used~\eqref{assumption:angles} to get that $|\theta(k_0)|+\frac{4\pi}{3} \leq |\theta(i_0)|+\frac{1}{2}+\frac{4\pi}{3} \leq |\theta(i_0)| +  2 \pi \leq 2 \pi (m+1)$. 
    \end{proof}
We are now in a position to prove Proposition~\ref{prop:psi is phi} and thus conclude the proof of the lower bound in Proposition~\ref{prop:lowerbound}.

\begin{proof}[Proof of Proposition~\ref{prop:psi is phi}] For the reader's convenience, we recall here the definitions of $\varphi(\nu)$ and $\psi(1,1,\nu)$:
    \begin{align*} 
        \varphi(\nu) & = \lim_{\varepsilon \to 0} \min\{ F_\e(u,Q^\nu) \colon u = u^\pos_\e \text{ on } \partial_\e^+ Q^\nu  \text{ and } u = u^\neg_\e \text{ on } \partial_\e^- Q^\nu \} \, , \\
        \psi(1,1,\nu)& = \inf\big\{\liminf_{\e\to 0}F_\e(u_\e,Q^\nu)\colon\chi(u_\e)\to\chi_\nu\ \text{in}\ L^1(Q^\nu)\big\}\, .
    \end{align*}
    Let us fix a sequence $(u_\e)$ such that $\chi(u_\e)\to\chi_\nu$ in $L^1(Q^\nu)$ and $F_\e(u_\e,Q^\nu) \to \psi(1,1,\nu)$. The aim of this proof is to define a modification $\tilde u_\e$ of $u_\e$ such that 
    \begin{gather}
        \tilde u_\e = u^\pos_\e \text{ on } \partial_\e^+ Q^\nu  \text{ and } \tilde u_\e = u^\neg_\e \text{ on } \partial_\e^- Q^\nu  ,  \label{condition:bc}\\
        \limsup_{\e \to 0} F_\e(\tilde u_\e,Q^\nu) \leq \lim_{\e \to 0} F_\e(u_\e,Q^\nu)\, . \label{condition:energy estimate}
    \end{gather}
    This allows us to conclude that $\varphi(\nu) \leq \psi(1,1,\nu)$. 

    The construction of the modified sequence $(\tilde u_\e)$ is divided in several steps.

    \begin{step}{1} (Choosing a strip with low energy). 
        We begin the construction by exploiting the property that the energy of $(u_\e)$ concentrates close to the interface $Q^\nu\cap L^\nu$ in order to choose a strip with low energy. Given $\delta \in (0,\tfrac{1}{3})$, we consider the family of strips $\mathscr{S}^\nu_{\e,\delta}$ defined in~\eqref{def:strips} and we apply Lemma~\ref{lemma:choosing strip} to deduce the existence of a strip $S_\e = S_{\e,r_\e} = Q_{r_\e+12\e}^\nu\setminus \big(\overline{Q}_{r_\e}^\nu\cup\overline{R}_{1,\delta}^\nu\big)\in \mathscr{S}^\nu_{\e,\delta}$ such that 
        \begin{equation}  \label{est:energy on slice}
            F_\e(u_\e,S_\e)  + \|\chi(u_\e)-\chi_\nu\|_{L^1(S_\e)} \leq \e \sigma_\e  \, ,
        \end{equation}
            where $\sigma_\e \to 0$.  The modification $\tilde u_\e$ of $u_\e$ will coincide with~$u^\pos_\e$ and~$u^\neg_\e$ in~$Q^{\nu} \sm (\overline Q^\nu_{1-\delta} \cup \overline R^{\nu}_{1,\delta})$ (notice that the square $Q^\nu_{1-\delta}$ contains  the closure of $S_\e$, \cf~\eqref{def:strips}).    In the triangles contained in~$S_\e$ the energy is low and thus $u_\e$ is close to ground states, yet not necessarily $u_\e^\pos$ or $u_\e^\neg$. There~$\tilde u_\e$ will start to interpolate from the configuration $u_\e$ until it reaches the fixed ground state $u_\e^\pos$ or~$u_\e^\neg$ close to the boundary.

            We shall describe in detail how to define $\tilde u_\e$ in the top part of the cube given by $Q_+^\nu = Q^\nu \cap \{x \colon \langle x, \nu \rangle > 0\}$, where the chirality of $u_\e$ converges to $1$. The construction in $Q_-^\nu \cap \{x \colon \langle x, \nu \rangle < 0\}$ is completely analogous.
    \end{step}

    \begin{step}{2} (Choosing triangles with low energy). We show here how to choose the triangles with low energy where to start the modification of $u_\e$. Let us consider the line  
        \begin{equation*} 
            L_\e := \{ x \in \R^2 \colon \langle x , \nu \rangle = \tfrac{r_\e}{2} + 3\e \} \, ,
        \end{equation*}
         which cuts in two the top part of the strip given by the rectangle
         \begin{equation} \label{def:top}
            S_\e^{\mathrm{top}} := R_{r_\e,6\e}^\nu\big((\tfrac{r_\e}{2}+3\e)\nu\big)= \big(L_\e + B_{3\e}(0)\big) \cap R^\nu_{r_\e,1} \subset S_\e\, .
         \end{equation}
          We describe now how to start the modification in $S_\e^{\mathrm{top}}$. The modification in the other parts 
          \begin{align*}
            S^{\mathrm{left}}_\e & := R_{r_\e,6\e}^{\nu^\perp}\big((\tfrac{r_\e}{2}+3\e)\nu^\perp\big)\sm \overline R^\nu_{1,\delta}\, ,\\
	S^{\mathrm{right}}_\e & := R_{r_\e,6\e}^{\nu^\perp}\big(-(\tfrac{r_\e}{2}+3\e)\nu^\perp\big)\sm \overline R^\nu_{1,\delta} \, ,
        \end{align*}
        \cf Figure~\ref{fig:top}, will be only sketched  since it is completely analogous. 
\begin{figure}[H]
   \includegraphics{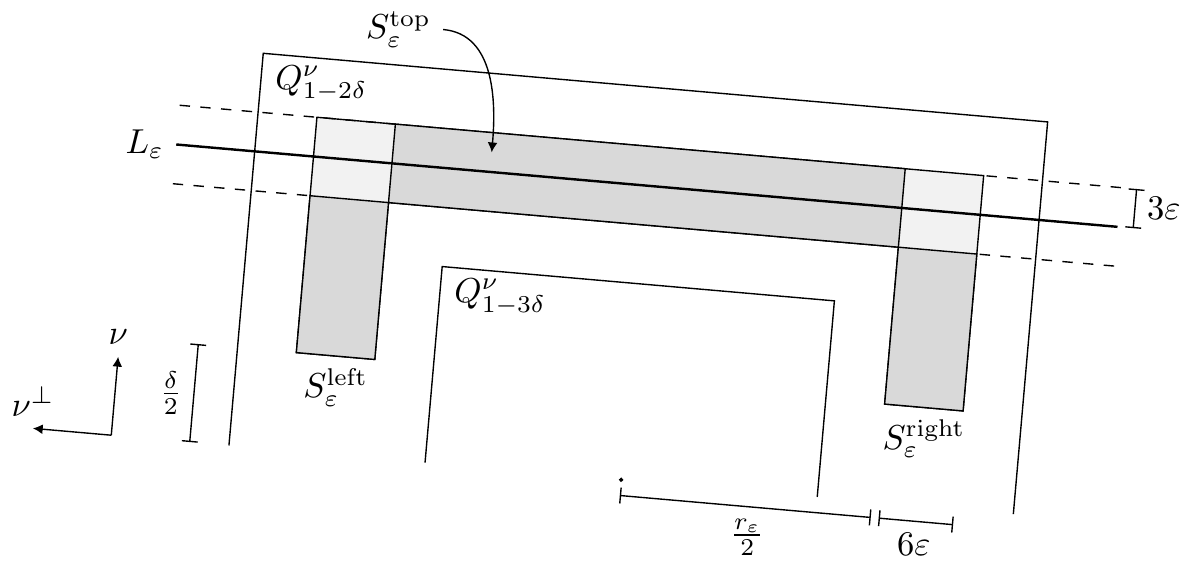}
\caption{Parts of the strip $S_\e$ in $Q^\nu_+$. }
\label{fig:top}
\end{figure}

    We consider now the slices $(\Sigma^{\alpha,z}_\e)_{z \in \Z}$ of the $\e$-triangular lattice defined in~\eqref{def:slice}. We choose $\alpha \in \{1,2,3\}$ such that $|\langle \tb_\alpha , \nu \rangle | \geq \frac{\sqrt{3}}{2}$, namely the best approximation of $\nu$ in the set~$\{\tb_1,\tb_2,\tb_3\}$. Equivalently, $|\langle \tb_\alpha , \nu^\perp \rangle | \leq \frac{1}{2}$, where $\nu^\perp$ is the direction of $L_\e$. (For $S^{\mathrm{right}}_\e$ and $S^{\mathrm{left}}_\e$ we consider a different direction, namely $\beta \in \{1,2,3\}$ such that $|\langle \tb_\beta , \nu^\perp \rangle | \geq \frac{\sqrt{3}}{2}$.)

        We can find a chain of closed triangles which intersect $L_\e$ such that each slice in the direction~$\tb_\alpha$ contains only one triangle of the chain. Specifically, there exist $(T_z)_{z \in \Z}$, satisfying 
        \begin{equation} \label{eq:properties chain}
            T_z \in \mathcal{T}^+_\e(\R^2)  \, , \quad  T_z \subset \Sigma^{\alpha,z}_\e , \quad  T_z \cap L_\e \neq \emptyset \, , \quad T_z \cap T_{z+1} \neq \emptyset \, ,
        \end{equation}
        for every $z \in \Z$, \cf Figure~\ref{fig:chain}. We prove this statement in Lemma~\ref{lemma:chain of triangles} below, since the geometric argument is irrelevant for the present discussion.  
\begin{figure}[H]
    \includegraphics{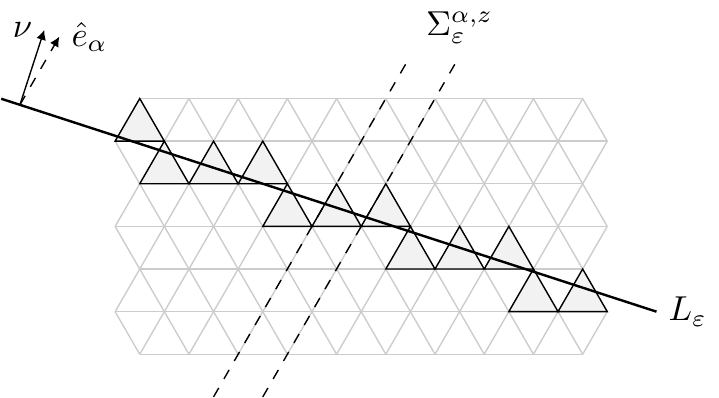}
\caption{A chain of triangles $(T_z)_{z \in \Z}$.}
\label{fig:chain}
\end{figure}

    The modification of $u_\e$ starts in the triangles~$T_z$ of the chain contained in~$S_\e^{\mathrm{top}}$. For this reason it is convenient to consider 
    \begin{equation*}
        \mathcal{Z}^{\mathrm{top}}_\e := \{ z \in \Z  \colon T_z \subset S_\e^{\mathrm{top}} \} \quad\text{and}\quad z_0 \in \argmin \mathcal{Z}^{\mathrm{top}}_\e\, .
    \end{equation*}
For future purposes we observe that
    \begin{equation} \label{est:number of triangles}
        \frac{\sqrt{3}}{4} \e^2 \# \mathcal{Z}^{\mathrm{top}}_\e = \Big|   \bigcup_{z \in \mathcal{Z}^{\mathrm{top}}_\e } T_z  \Big| \leq |S_\e^{\mathrm{top}}| =  6 \e r_\e \leq 6\e \quad \implies  \quad \# \mathcal{Z}^{\mathrm{top}}_\e \leq  \frac{C_1}{\e} \, ,
    \end{equation}
    for some positive constant $C_1$ and for $\e$ small enough.
\end{step}

\begin{step}{3} (Estimating the maximal winding number). The energy regime we are working in does not rule out the possibility that inside the strip $S_\e$ the configuration of spin field displays global rotations. However, the bound of the energy in $S_\e$ allows us to estimate the maximal number of complete turns of $2\pi$. To present precisely the estimate, we define in the triangles chosen in Step~2 the liftings $\theta_\e \in \R$  of $u_\e$ according to the following recursive argument. Given~$z \in \mathcal{Z}^{\mathrm{top}}_\e$ we denote by $i_z \in \L^1$, $j_z \in \L^2$, $k_z \in \L^3$ the points in the sublattices such that $\e i_z, \e j_z, \e k_z$ are the vertices of the triangle $T_z$ (some points might have multiple labels). We now define recursively angles $\theta_\e(\e i_z), \theta_\e(\e j_z), \theta_\e(\e k_z)$ in suitably chosen intervals of length $2\pi$ satisfying $u_\e(\e i_{z}) = \exp\big(\iota \theta_\e (\e i_{z})\big), u_\e(\e j_{z}) = \exp\big(\iota \theta_\e (\e j_{z})\big), u_\e(\e k_{z}) = \exp\big(\iota \theta_\e (\e k_{z})\big)$ as follows. We choose
\begin{equation*}
\begin{split}
& \theta_\e (\e i_{z_0}) \in [0,2\pi)\, ,\\
& \theta_\e(\e j_{z_0})\in[\theta_\e(\e i_{z_0})-\pi,\theta_\e(\e i_{z_0})+\pi)\, ,\\
& \theta_\e(\e k_{z_0}) \in [\theta_\e (\e j_{z_0}) - \pi, \theta_\e (\e j_{z_0}) + \pi)\, ,\\
& \theta_\e (\e i_{z+1}) \in [\theta_\e (\e i_{z}) - \pi, \theta_\e (\e i_{z}) + \pi)\, .
\end{split}
\end{equation*}
The choice of $\theta_\e(\e j_z)$ and $\theta_\e(\e k_z)$ is made according to the same recursive procedure above, but taking as starting point (instead of   
$\theta_\e (\e i_{z_0})$) the angles $\theta_\e(\e j_{z_0})$ and $\theta_\e(\e k_{z_0})$, respectively. 
    We claim that  
    \begin{equation} \label{est:maximal m}
    \frac{1}{2\pi}\sup_{\substack{ z \in \mathcal{Z}^{\mathrm{top}}_\e \\ z \geq z_0 }} \big\{ |\theta_\e(\e i_z) - \theta_\e(\e i_{z_0})| \, ,  \ |\theta_\e(\e j_z) - \theta_\e(\e j_{z_0})| \, , \ |\theta_\e(\e k_z) - \theta_\e(\e k_{z_0})|   \big\} \leq C_2 \sqrt{\frac{\sigma_\e}{\e}}\, ,
    \end{equation}
    for some positive constant $C_2$. To prove the claim, let us fix $z_* \in \mathcal{Z}^{\mathrm{top}}_\e$, $z_* \geq z_0$. Note that $z_* - z_0 \leq \frac{C_1}{\e}$ by~\eqref{est:number of triangles}. Jensen's inequality implies that
    \begin{equation} \label{est:angle difference}
        \begin{split}
            |\theta_\e(\e i_{z_*}) - \theta_\e(\e i_{z_0})|^2 & \leq \Big( \sum_{z=z_0}^{z_*-1} |\theta_\e(\e i_{z+1}) - \theta_\e(\e i_z)| \Big)^{\! 2} \leq (z_*-z_0) \sum_{z=z_0}^{z_*-1} |\theta_\e(\e i_{z+1}) - \theta_\e(\e i_z)|^2 \\
            & \leq \frac{C_1}{\e} \sum_{z=z_0}^{z_*-1} |\theta_\e(\e i_{z+1}) - \theta_\e(\e i_z)|^2 \leq \frac{C}{\e} \sum_{z=z_0}^{z_*-1} 2-2\cos\big(\theta_\e(\e i_{z+1}) - \theta_\e(\e i_z)\big)  \\
            & = \frac{C}{\e} \sum_{z=z_0}^{z_*-1} |u_\e(\e i_{z+1}) - u_\e(\e i_z)|^2   
        \end{split}
    \end{equation}
    for some positive constants $C$, where we used the fact that $1 - \cos(\phi) \geq \frac{1}{12} \phi^2$ for every $|\phi| \leq \pi$. We start observing that the regular hexagon $H_z$ containing $T_z$ and $T_{z+1}$ is contained in $S_\e$. Indeed, let $x \in H_z$ and let $y \in T_z \cap L_\e \subset S_\e^\mathrm{top}$. Then $\dist(x,L_\e) \leq |x - y| \leq \diam H_z = 2 \e < 3  \e$. Hence, \cf \eqref{def:top}, $x \in (L_\e + B_{3\e}(0)) \cap R^\nu_{r_\e + 6 \e,1} \subset S_\e$. Let us show that 
   \begin{equation} \label{est:with hexagon}
       |u_\e(\e i_{z+1}) - u_\e(\e i_{z})|^2 \leq \frac{2}{\e}   F_\e(u_\e, H_z) \, .
   \end{equation} 
  Indeed, if $T_z \cap T_{z+1} = \{\e i_z\} = \{\e i_{z+1}\}$, then $|u_\e(\e i_{z+1}) - u_\e(\e i_{z})|^2 = 0$; if $T_z \cap T_{z+1} = \{\e j_z\} = \{\e j_{z+1}\}$ (and analogously if $T_z \cap T_{z+1} = \{\e k_z\} = \{\e k_{z+1}\}$), then we let $T'$ be the third triangle in $\conv\{T_z, T_{z+1}\}$. The triangle $T'$ is either $\conv\{\e i_z, \e j_z, \e k_{z+1}\}$ or $\conv\{\e i_{z+1}, \e j_z, \e k_z\}$ and is always contained in $H_z$, see Figure~\ref{fig:hexagon}. Letting $\e k$ be its vertex in $\L^3_\e$ (either $\e k_z$ or $\e k_{z+1}$) we have that
    \begin{equation*}  
        \begin{split}
            |u_\e(\e i_{z+1}) - u_\e(\e i_{z})|^2 & \leq 2 \, |u_\e(\e i_{z+1}) + u_\e(\e j_{z+1}) + u_\e(\e k)|^2  + 2 \, |u_\e(\e i_{z}) + u_\e(\e j_{z}) + u_\e(\e k)|^2 \\
            & \leq \frac{2}{\e}   F_\e(u_\e, H_z) \, .
        \end{split}
    \end{equation*} 

    \begin{figure}[H]
        \includegraphics{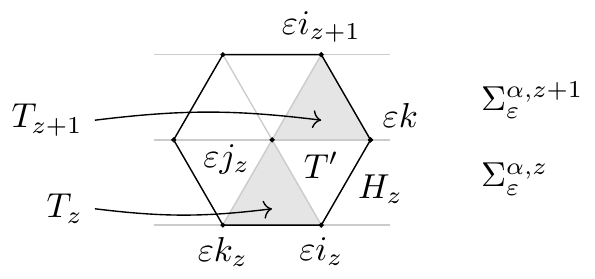}
        \caption{Triangle $T'$ in a possible configuration of $T_z$ and $T_{z+1}$.}
        \label{fig:hexagon}
    \end{figure}

    Then we estimate the last sum in~\eqref{est:angle difference} using~\eqref{est:with hexagon} by
    \begin{equation*}
        \sum_{z=z_0}^{z_*-1} |u_\e(\e i_{z+1}) - u_\e(\e i_{z})|^2 \leq \sum_{z=z_0}^{z_*-1} \frac{2}{\e} F_\e(u_\e,H_{z}) \leq \frac{C}{\e}  F_\e(u_\e,S_\e)  \leq C \sigma_\e \, ,
    \end{equation*}
    for some positive constant $C$. In conclusion, by~\eqref{est:angle difference} we have that 
    \begin{equation*}
        |\theta_\e(\e i_{z_*}) - \theta_\e(\e i_{z_0})| \leq C \sqrt{\frac{\sigma_\e}{\e}} \, .
    \end{equation*}
    
    Arguing in an analogous way for $|\theta_\e(\e j_{z_*}) - \theta_\e(\e j_{z_0})|$ and $|\theta_\e(\e k_{z_*}) - \theta_\e(\e k_{z_0})|$, we conclude the proof of the claim~\eqref{est:maximal m}.  

    We consider the bound on the maximal winding number given by 
    \begin{equation} \label{def:winding number}
        m_\e := \Big\lceil \,  C_2 \sqrt{\frac{\sigma_\e}{\e}} \, \Big\rceil  + 8\, ,
    \end{equation}
    where $\lceil \,  C_2 \sqrt{\frac{\sigma_\e}{\e}} \, \rceil$ is the smallest natural number grater than or equal to $ C_2 \sqrt{\frac{\sigma_\e}{\e}}$ and $C_2$ is the constant given in~\eqref{est:maximal m}. 
    \end{step}

\begin{step}{4} (Modification on slices). We define the modification on the slices $\Sigma^{\alpha,z}_\e$ starting from triangles $T_z$ with $z \in \mathcal{Z}^{\mathrm{top}}_\e$ by reproducing the construction in Lemma~\ref{lemma:1d interpolation}. Here we make precise the choice of the parameters for this construction and the notation. Let us assume, without loss of generality, that $\langle \tb_\alpha , \nu \rangle  \geq \frac{\sqrt{3}}{2}$ (if, instead, $\langle \tb_\alpha , \nu \rangle   \leq - \frac{\sqrt{3}}{2}$ we work with $- \tb_\alpha$). For $z\in \mathcal{Z}^{\mathrm{top}}_\e$ we let $i_z^0 := i_z \in \L^1$, $j_z^0 := j_z \in \L^2$, $k_z^0 := k_z \in \L^3$ where $\e i_z$, $\e j_z$, $\e k_z$ are the vertices of $T_z$. As in~\eqref{def:vertices in slice}, we define the lattice points $i_z^h \in \L^1$, $j_z^h \in \L^2$, $k_z^h\in\L^3$ and the triangle $T_z^h$ with the following recursive formula: for $h\in\N$ we set
    \begin{equation} \label{def:vertices in slice eps}
        \begin{split}
            i_z^{h+1} &:= i_z^h +   \tb_\alpha + \tfrac{1}{2} \tb_\alpha + \tau(\langle i_z^h , \tb_\alpha^\perp \rangle) \tfrac{\sqrt{3}}{2} \tb_\alpha^\perp \, ,\\
            j_z^{h+1} &:= j_z^h + \tb_\alpha + \tfrac{1}{2} \tb_\alpha + \tau(\langle j_z^h , \tb_\alpha^\perp \rangle) \tfrac{\sqrt{3}}{2} \tb_\alpha^\perp \, ,\\
            k_z^{h+1} &:= k_z^h + \tb_\alpha + \tfrac{1}{2} \tb_\alpha + \tau(\langle k_z^h , \tb_\alpha^\perp \rangle) \tfrac{\sqrt{3}}{2} \tb_\alpha^\perp \, , \\
            T_z^{h+1} & := \conv\{\e i_z^{h+1},\e j_z^{h+1}, \e k_z^{h+1}\} \subset \Sigma^{\alpha,z}_\e \, ,
        \end{split}
    \end{equation}
    where $\tau(0) := 1$, $\tau(\tfrac{\sqrt{3}}{2}) := -1$. As in~\eqref{def:lattice slice}, we define the half-slice $\Sigma^{\alpha,z}_\e(T_z)$ of the lattice $\L_\e$ starting from $T_z$ by $\Sigma^{\alpha,z}_\e(T_z) \defas\conv\{T_z^h \colon h \in \N \}$.

    \noindent {\em Number of interpolation steps}. The number of interpolation steps will be defined by finding the first shifted triangle $T_{z_0}^{2 h}$ in the half-slice~$\Sigma^{\alpha,z_0}_\e(T_{z_0})$ that is well contained in $R^\nu_{\infty,1-\delta} \sm \overline R^{\nu}_{\infty, 1-2\delta}$. Specifically, we define 
    \begin{equation*}
       N_\e := \min \{ 2h  \colon h \in \N \, , \ T_{z_0}^{2 h} \subset R^\nu_{\infty, 1-5\delta/4} \sm \overline R^{\nu}_{\infty, 1-7 \delta/4} \} \, .
    \end{equation*}  
     Given another $z \in \mathcal{Z}^{\mathrm{top}}_\e$, we have that 
     \begin{equation} \label{eq:far triangle}
        T_{z}^{N_\e} \subset R^\nu_{\infty,1- \delta } \sm \overline R^{\nu}_{\infty, 1-2\delta } \, .
     \end{equation}
     Indeed, let $y = y_0 + 3\e \frac{N_\e}{2}  \tb_\alpha \in T_{z}^{N_\e}$ with $y_0 \in T_z$. Let $x_0 \in T_{z_0} \cap L_\e$, \cf~\eqref{eq:properties chain}, and let $x:= x_0+3\e \frac{N_\e}{2}  \tb_\alpha \in T_{z_0}^{N_\e}$. Since $y_0 \in L_\e + B_{3\e}(0)$, we have that $|\langle y_0-x_0,\nu \rangle| < 3 \e$ and thus $|\langle y-x,\nu \rangle| < 3 \e$, \ie $y$ belongs to the $3 \e$-neighborhood of $R^\nu_{\infty, 1-5\delta/4} \sm \overline R^{\nu}_{\infty, 1-7 \delta/4}$, which is contained in $R^\nu_{\infty,1- \delta } \sm \overline R^{\nu}_{\infty, 1-2\delta }$.   

    Observe that 
    \begin{equation} \label{est:N}
        N_\e \leq \frac{C_3}{\e}
    \end{equation}
    for some positive constant $C_3$. To prove this, let $x_0 \in T_{z_0}$ and $x := x_0+3\e \frac{N_\e}{2}  \tb_\alpha \in T_{z_0}^{N_\e}$. The segment $[x_0;x]$ is fully contained in $R^\nu_{\infty,1-\delta} \sm \overline R^\nu_{\infty,1-3\delta}$ and thus $\delta  \geq |\langle x-x_0 , \nu \rangle| =  3\e \frac{N_\e}{2} \langle \tb_\alpha , \nu \rangle  \geq 3\e \frac{N_\e}{2} \frac{\sqrt{3}}{2}$.

    \noindent {\em Winding number}. We choose $m_\e$ given by~\eqref{def:winding number}. We consider the angles $\theta_\e(\e i_{z})$, $\theta_\e(\e j_{z})$, $\theta_\e(\e k_{z})$ introduced in Step~3. By~\eqref{def:winding number} and \eqref{est:maximal m} we infer that 
    \begin{equation*}
        2 \pi m_\e  \geq 2 \pi C_2 \sqrt{\frac{\sigma_\e}{\e}} + 16 \pi \geq 2 \pi |\theta_\e(\e i_z) - \theta_\e(\e i_{z_0})| + 16 \pi \geq 2 \pi |\theta_\e(\e i_z)| + 2 \pi \, ,
    \end{equation*} 
    hence~\eqref{assumption:m} is satisfied.

    \noindent {\em Checking the assumptions on the angles}. We check that the assumptions~\eqref{assumption:angles} are satisfied. First, we claim that for $\e$ small enough the configuration $u_\e$ has positive chirality in every triangle $T \in \mathcal{T}_\e(\R^2)$ contained in $S_\e^{\mathrm{top}}$. To prove it, let us start by showing that the sign of the chirality is constant arguing by contradiction. Assume that there exist two triangles $T',T'' \subset S_\e^{\mathrm{top}}$ with a common side such that $\chi(u_\e) \leq 0$ in $T'$ and $\chi(u_\e) \geq 0$ in $T''$. Then by~\eqref{est:energy on slice} and Lemma~\ref{lemma:energy of opposite chirality} we would get 
    \begin{equation*}
        \e \sigma_\e \geq F_\e(u_\e,S_\e) \geq F_\e(u_\e, T' \cup T'') \geq  \frac{5}{3} \e   \, ,  
    \end{equation*}
    which contradicts the condition $\sigma_\e \to 0$. Therefore $\chi(u_\e)$ has constant sign in $S_\e^{\mathrm{top}}$. In fact, $\chi(u_\e) > 0$ in $S_\e^{\mathrm{top}}$. If instead $\chi(u_\e) \leq 0$ in $S_\e^{\mathrm{top}}$, by~\eqref{est:energy on slice} we would have that  
    \begin{equation*}
       \e \sigma_\e \geq \|\chi(u_\e) - \chi_\nu\|_{L^1(S_\e^{\mathrm{top}})} = \int_{S_\e^{\mathrm{top}}} (1 - \chi(u_\e) ) \d x \geq |S_\e^{\mathrm{top}}| = 6 \e r_\e  \geq 6 \e \big(\tfrac{1}{2}-\tfrac{3}{2}\delta\big) \, ,
    \end{equation*}
    which contradicts $\sigma_\e \to 0$.  In conclusion, $\chi(u_\e) > 0$ in $S_\e^{\mathrm{top}}$. 

    Let now $z \in \mathcal{Z}^{\mathrm{top}}_\e$. We have
    \begin{equation*}
       |u_\e(\e i_z) + u_\e(\e j_z) + u_\e(\e k_z)|^2 = \frac{1}{\e}F_\e(u_\e,T_z) \leq  \frac{1}{\e} F_\e(u_\e,S_\e) \leq \sigma_\e \,.
    \end{equation*}
    Since $\chi(u_\e) > 0$ in $T_z$, for $\e$ small enough $u_\e$ is close to a ground state with chirality 1 and therefore, using~\eqref{F:angularlift} and Lemma~\ref{lem:chirality} (see also~\eqref{eq:wells}),
    \begin{equation} \label{assumption:angles eps}
        \Big|\theta_\e(\e j_z) - \theta_\e(\e i_z) - \frac{2\pi}{3}\Big| \leq \frac{1}{4} \, , \quad \Big|\theta_\e(\e k_z) - \theta_\e(\e j_z)- \frac{2\pi}{3}\Big| \leq \frac{1}{4} \, .
    \end{equation}


    \noindent {\em Definition of the interpolation}. We are in a position to define the interpolation. We reproduce the one-dimensional construction of Lemma~\ref{lemma:1d interpolation} by suitably translating and scaling it, providing the precise notation as it will be useful for later estimates. We shall define the interpolation only on slices starting from every other triangle~$T_z$, for the constructions on two slices $\Sigma^{\alpha,z}_\e$ and $\Sigma^{\alpha,z+2}_\e$ completely determine the values of the modified spin configuration in $\Sigma^{\alpha,z+1}_\e$. For this reason, let $z \in \mathcal{Z}^{\mathrm{top}}_\e$ be such that $z \equiv z_0 \!\! \mod 2$. We then define the interpolated angles $\theta(\e i_z^h)$, $\theta(\e j_z^h)$, $\theta(\e k_z^h)$ for $h=0,\dots,N_\e$ as in~\eqref{def:interpolated angles} by (recall that $i_z^0 = i_z$, $j_z^0 = j_z$, $k_z^0 = k_z$)
    \begin{equation} \label{def:interpolated angles eps}
        \begin{split}
            \theta_\e(\e i_z^h) & \defas\theta_\e(\e i_z)+h\frac{2\pi m_\e-\theta_\e(\e i_z)}{N_\e} = \Big(1 - \frac{h}{N_\e}\Big)\theta_\e(\e i_z) + \frac{h}{N_\e}2\pi m_\e  \, , \\
            \theta_\e(\e j_z^h) & \defas\theta_\e(\e j_z)+h\frac{2\pi m_\e+\frac{2\pi}{3}-\theta_\e(\e j_z)}{N_\e} = \Big(1 - \frac{h}{N_\e}\Big)\theta_\e(\e j_z) + \frac{h}{N_\e}2\pi m_\e + \frac{h}{N_\e}\frac{2\pi}{3} \, , \\
            \theta_\e(\e k_z^h) & \defas\theta_\e(\e k_z)+h\frac{2\pi m_\e+\frac{4\pi}{3}-\theta_\e(\e k_z)}{N_\e} =  \Big(1 - \frac{h}{N_\e}\Big)\theta_\e(\e k_z) + \frac{h}{N_\e}2\pi m_\e + \frac{h}{N_\e}\frac{4\pi}{3}\, , \\
        \end{split}
    \end{equation}
    and $\theta_\e(\e i_z^h) := 2\pi m_\e$, $\theta_\e(\e j_z^h) := 2\pi m_\e+\frac{2\pi}{3}$, $\theta_\e(\e k_z^h) := 2\pi m_\e+\frac{4\pi}{3}$ for $h \geq N_\e+1$. 
    Eventually, we define $u_\e^\mathrm{top}\colon \L_\e \cap\Sigma^{\alpha,z}_\e(T_z) \to\S^1$ by setting
    \begin{equation*} 
        u_\e^\mathrm{top}(\e i_z^h) \defas\exp(\iota \theta(\e i_z^h))\, ,\quad u_\e^\mathrm{top} (\e j_z^h)\defas\exp(\iota \theta(\e j_z^h))\, ,\quad u_\e^\mathrm{top} (\e k_z^h)\defas\exp(\iota \theta(\e k_z^h))\, .
    \end{equation*}
    By~\eqref{eq:far triangle} we have that 
    \begin{equation} \label{eq:attained boundary conditions}
        u_\e^\mathrm{top}|_T = u_\e^\pos|_T \quad \text{if } T \subset \Sigma^{\alpha,z}_\e(T_z)  \sm \overline R^\nu_{\infty,1-\delta}  \, .
    \end{equation}

    \noindent {\em Estimate on ``even'' slices.}    We observe that the construction of $u_\e^\mathrm{top}$ is simply a translation and a scaling of the construction in Lemma~\ref{lemma:1d interpolation}. As the assumption~\eqref{assumption:angles} is satisfied, \cf~\eqref{assumption:angles eps}, we can apply Lemma~\ref{lemma:1d interpolation} to deduce that 
    \begin{equation} \label{est:on even}
        F_\e(u^\mathrm{top}_\e, \Sigma^{\alpha,z}_\e(T_z)) \leq C \Big(N_\e F_\e(u_\e,T_z)+ \e \frac{m_\e^2}{N_\e}\Big)  \, .
    \end{equation}

    \noindent {\em Estimate on ``odd'' slices.} We estimate the energy on the missing half-slices. Let us fix $z, z+1, z+2 \in \mathcal{Z}^{\mathrm{top}}_\e$ with $z \equiv z_0 \! \mod 2$. Let $T$ be a triangle contained in $\Sigma^{\alpha,z+1}_\e(T_{z+1})$. Then $T$ shares two vertices with one triangle contained in $\Sigma^{\alpha,z}_\e(T_{z})$ or with one triangle contained in $\Sigma^{\alpha,z+2}_\e(T_{z+2})$. Let us assume, without loss of generality, that the two shared vertices are the vertices $\e j'\in \L^2_\e$ and $\e k' \in \L^3_\e$ of some triangle $T' \subset \Sigma^{\alpha,z}_\e(T_{z})$. The third vertex of $T'$ is of the type $\e i_z^{h'} \in \L^1_\e$ for some $h' \in \N$.  Moreover, the third vertex of~$T$ is shared with a triangle $T_{z+2}^{h}$, $h \in \N$, and is of the type $\e i_{z+2}^{h}\in \L^3_\e$\EEE. We remark that $|h'-h| \leq 2$. Indeed, by~\eqref{def:vertices in slice eps} we have that 
    \begin{align*}
        i_z^{h'} & = i_z + h' \tfrac{3}{2} \tb_\alpha \pm \tfrac{\sqrt{3}}{2} \tb_\alpha^\perp \, , \\
        i_{z+2}^{h} & = i_{z+2} + h \tfrac{3}{2} \tb_\alpha \pm \tfrac{\sqrt{3}}{2} \tb_\alpha^\perp  \, .
    \end{align*} 
From the assumptions on the position of the two triangles together with the definition of $i_z^{h'},  i_{z+2}^{h}$ it follows that
    \begin{equation*}
        0 = \langle i_z^{h'} - i_{z+2}^{h} , \tb_\alpha \rangle = \langle i_z - i_{z+2} , \tb_\alpha \rangle + (h'-h)\tfrac{3}{2}  \quad \implies \quad |h' - h| = \tfrac{2}{3} |\langle i_z - i_{z+2} , \tb_\alpha \rangle| \leq 2 \, ,
    \end{equation*}
    where in the last inequality we used the fact that $T_z \cap T_{z+1} \neq \emptyset$ and $T_{z+1} \cap T_{z+2} \neq \emptyset$.

    We estimate the energy in the triangle $T$ by
    \begin{equation*}
        \begin{split} 
            F_\e(u^\mathrm{top}_\e, T) & = \e |u^\mathrm{top}_\e(\e i_{z+2}^{h}) + u^\mathrm{top}_\e(\e j') + u^\mathrm{top}_\e(\e k')|^2 \\
            & \leq 2 \e  |u^\mathrm{top}_\e(\e i_{z}^{h'}) + u^\mathrm{top}_\e(\e j') + u^\mathrm{top}_\e(\e k')|^2 + 2 \e |u^\mathrm{top}_\e(\e i_{z}^{h'}) - u^\mathrm{top}_\e(\e i_{z+2}^{h})|^2 \\
            & = 2 F_\e(u^\mathrm{top}_\e, T') +  2 \e |u^\mathrm{top}_\e(\e i_{z}^{h'}) - u^\mathrm{top}_\e(\e i_{z+2}^{h})|^2 \, .
        \end{split}
    \end{equation*}
    Note that~\eqref{est:maximal m} and \eqref{def:winding number} imply
    \begin{equation*}
        \begin{split}
            |\theta_\e(\e i_{z+2})| & \leq |\theta_\e(\e i_{z+2}) - \theta_\e(\e i_{z_0}) | + | \theta_\e(\e i_{z_0})|  \leq |\theta_\e(\e i_{z+2}) - \theta_\e(\e i_{z_0}) | + 2\pi \leq 2\pi m_\e   \, .
        \end{split}
    \end{equation*}
    From~\eqref{def:interpolated angles eps}, from the previous estimate, and since $|h'-h|\leq 2$ it follows that 
    \begin{equation*}
        \begin{split}
            |u^\mathrm{top}_\e(\e i_{z}^{h}) - u^\mathrm{top}_\e(\e i_{z+2}^{h'})|^2 & = 2 - 2 \cos(\theta_\e(\e i_z^h) - \theta_\e(\e i_{z+2}^{h'})) \leq \big| \theta_\e(\e i_z^h) - \theta_\e(\e i_{z+2}^{h'}) \big|^2 \\
            &   = \Big|\Big(1 - \frac{h}{N_\e}\Big)\big( \theta_\e(\e i_z) - \theta_\e(\e i_{z+2})\big) + \frac{h-h'}{N_\e}\Big( 2\pi m_\e - \theta_\e(\e i_{z+2})\Big)  \Big|^{ 2} \\
            &   \leq 2 \big| \theta_\e(\e i_z) - \theta_\e(\e i_{z+2})\big|^2 + 2 \Big|\frac{h-h'}{N_\e}\Big|^2 \Big|  2\pi m_\e - \theta_\e(\e i_{z+2})  \Big|^{ 2} \\
            &  \leq 2 \big| \theta_\e(\e i_z) - \theta_\e(\e i_{z+2})\big|^2 + C \frac{m_\e^2}{N_\e^2} \, \\
        \end{split}
    \end{equation*}
    It remains to estimate $\big| \theta_\e(\e i_z) - \theta_\e(\e i_{z+2})\big|^2$. Using the fact that $1 - \cos(\phi) \geq \frac{1}{12} \phi^2$ for every $|\phi| \leq \pi$ and by~\eqref{est:with hexagon} we obtain that 
   \begin{equation*}
    \begin{split}
        \big| \theta_\e(\e i_z)  - \theta_\e(\e i_{z+2})\big|^2 & \leq 2 \big| \theta_\e(\e i_z)  - \theta_\e(\e i_{z+1})\big|^2 + 2 \big| \theta_\e(\e i_{z+1})  - \theta_\e(\e i_{z+2})\big|^2 \\
        & \leq C \big| u_\e(\e i_z)  - u_\e(\e i_{z+1})\big|^2 + C \big| u_\e(\e i_{z+1})  - u_\e(\e i_{z+2})\big|^2 \\
        & \leq \frac{C}{\e} \big(F_\e(u_\e, H_z) + F_\e(u_\e, H_{z+1}) \big) \, ,
    \end{split}
   \end{equation*}
   where $H_z$ is an hexagon containing $T_z$ and $T_{z+1}$ and $H_{z+1}$ is an hexagon containing $T_{z+1}$ and $T_{z+2}$.
   In conclusion, we have that 
   \begin{equation*}
    F_\e(u^\mathrm{top}_\e, T) \leq C \Big( F_\e(u^\mathrm{top}_\e, T') + F_\e(u_\e, H_z) + F_\e(u_\e, H_{z+1}) +  \e \frac{m_\e^2}{N_\e^2} \Big) \, .
   \end{equation*}
    Summing over all triangles in $\Sigma^{\alpha,z}_\e(T_{z+1})$ (their number is $C N_\e$) we deduce that 
    \begin{equation} \label{est:on odd}
        \begin{split}
            & F_\e(u^\mathrm{top}_\e, \Sigma^{\alpha,z+1}_\e(T_{z+1})) \\
            & \quad \leq C \Big( F_\e(u^\mathrm{top}_\e, \Sigma^{\alpha,z}_\e(T_{z})) +  F_\e(u^\mathrm{top}_\e, \Sigma^{\alpha,z+2}_\e(T_{z+2}))  +  N_\e F_\e(u_\e, H_z) + N_\e F_\e(u_\e, H_{z+1}) +  \e \frac{m_\e^2}{N_\e} \Big) \, .
        \end{split}
    \end{equation}

    \noindent {\em Final estimate on top part.} By~\eqref{est:on odd}, \eqref{est:on even}, summing over $z$ and by \eqref{est:energy on slice}, \eqref{est:number of triangles}, \eqref{def:winding number}, and~\eqref{est:N} we conclude that\footnote{In this estimate it becomes evident that it was crucial to prove that the energy concentrates close to the interface. A classical averaging/slicing argument would only provide a bound on the strip $S_\e$ of the type $F_\e(u_\e,S_\e) \leq C \e$. This would not suffice to conclude that the modified sequence does not increase the energy, as the right-hand side in this estimate would end up to be a constant.}
    \begin{equation} \label{est:on top part}
        \begin{split}
            \sum_{z \in \mathcal{Z}^{\mathrm{top}}_\e} \hspace{-0.5em} F_\e(u^\mathrm{top}_\e, \Sigma^{\alpha,z}_\e(T_z)) & \leq \hspace{-0.8em} \sum_{\substack{z \in \mathcal{Z}^{\mathrm{top}}_\e \\ z \equiv z_0 \!\!\!\!\! \mod 2}} \hspace{-0.8em} C \Big(F_\e(u^\mathrm{top}_\e, \Sigma^{\alpha,z}_\e(T_z))+  N_\e F_\e(u_\e, H_z) + N_\e F_\e(u_\e, H_{z+1}) + \e \frac{m_\e^2}{N_\e}\Big) \\
            & \leq \hspace{-0.8em} \sum_{\substack{z \in \mathcal{Z}^{\mathrm{top}}_\e \\ z \equiv z_0 \!\!\!\!\! \mod 2}} \hspace{-0.8em} C \Big(N_\e F_\e(u_\e,T_z)+  N_\e F_\e(u_\e, H_z) + N_\e F_\e(u_\e, H_{z+1}) + \e \frac{m_\e^2}{N_\e}\Big) \\
            & \leq C N_\e F_\e(u_\e,S_\e) +  C \e \frac{m_\e^2}{N_\e} \#\mathcal{Z}^{\mathrm{top}}_\e  \leq C \frac{C_3}{\e} \e \sigma_\e + C \e \Big(C_2\sqrt{\frac{\sigma_\e}{\e}} + 4\Big)^{\!2} \frac{\e}{C_3}  \frac{C_1}{\e}  \\
            & \leq C\left(\sigma_\e +\e \right) \, .
        \end{split}
    \end{equation}
    \end{step}

    \begin{step}{5} (Definition of modification in remaining parts of the square). The modification starting from $S_\e^\mathrm{left}$ and $S_\e^\mathrm{right}$ is completely analogous. We recall that $\beta \in \{1,2,3\}$ is such that $|\langle \tb_\beta , \nu^\perp \rangle | \geq \frac{\sqrt{3}}{2}$. We consider chains of triangles contained in $S_\e^\mathrm{left}$ and $S_\e^\mathrm{right}$ given by Lemma~\ref{lemma:chain of triangles} (suitably adapted). In half-slices in the direction $\tb_\beta$ starting from triangles of these chains and approaching the boundary $\partial Q^\nu$, we define $u_\e^\mathrm{left}$ and $u_\e^\mathrm{right}$ as in Step 4. 

    \begin{figure}[H]
       \includegraphics{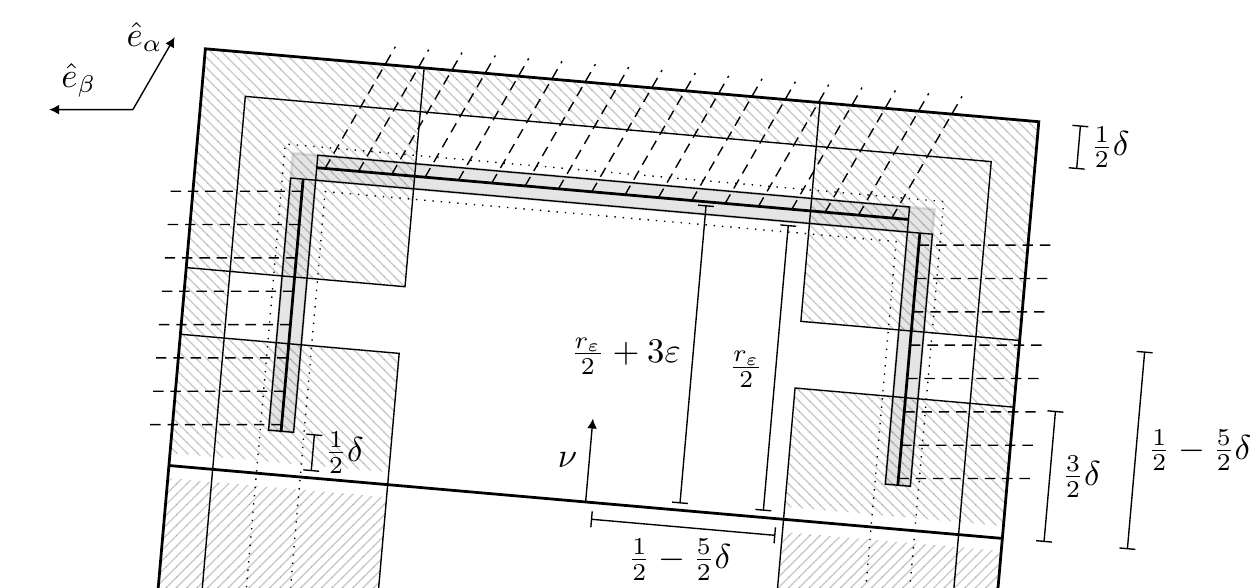} 
            \caption{Definition of $\tilde u_\e$ in $Q^\nu_+$: in the hatched regions it is equal to $u^\pos$; in the white region enclosed by $Q^\nu_{r_\e+6\e}$ it is equal to $u_\e$; outside of $Q^\nu_{r_\e}$ it is defined through the interpolation $u^\mathrm{top}_\e$, $u^\mathrm{left}_\e$, and $u^\mathrm{right}_\e$ constructed with the slices in the lattice directions $\tb_\alpha$ and~$\tb_\beta$.}
            \label{fig:global construction}
    \end{figure}

        We are finally in a position to define $\tilde u_\e$ in $Q^\nu_+$. We fix $\delta \in (0,\tfrac{1}{8})$ and we consider the two-barred cross-shaped set (the white region in Figure~\ref{fig:global construction})
        \begin{equation*}
            P_\delta :=    R^\nu_{1-5\delta,1} \cup (R^\nu_{1,1-5\delta} \sm \overline R^\nu_{1,3\delta})    \, .
        \end{equation*}
        Given $T \in \mathcal{T}_\e(\R^2)$ such that $T \subset Q^\nu$, we distinguish some cases.

        \noindent{\em Case} $T \subset P_\delta \cap  Q^\nu_{r_\e+6\e}$: We set 
        \begin{equation} \label{def:tilde u 1}
		\text{if } T \subset P_\delta \cap Q^\nu_{r_\e+6\e}: \,  \quad\tilde u_\e|_T := u_\e|_T 
        \end{equation}

        \noindent{\em Case} $T \subset R^\nu_{1-5\delta,1} \sm \overline Q^\nu_{r_\e}$ (part of the cross-shaped set $P_\delta$ aligned with $\nu$): We give the definition in the case $T \subset Q^\nu_+$ (the case $T \subset Q^\nu_-$ being analogous). Let $y_0 \in T$. Let us consider the slice~$\Sigma^{\alpha,z}_\e$ such that $T \subset \Sigma^{\alpha,z}_\e$ and let us show that $z \in \mathcal{Z}^{\mathrm{top}}_\e$. Let $x \in T_z$ and first of all note that $x \in L_\e + B_{3 \e}(0)$. Since  $T$ and $T_z$ are contained in the same slice, by definition of $\Sigma^{\alpha,z}_\e$ we can find $s \in \R$ such that $x_0 := y_0 + s \tb_\alpha \in T_z$. Since $y_0 \notin \overline Q^\nu_{r_\e}$, the segment $[x_0;y_0]$ is contained in $R^\nu_{\infty,1} \sm  \overline R^\nu_{\infty,r_\e}$, thus $|s| \frac{\sqrt{3}}{2} \leq |s| |\langle \tb_\alpha, \nu \rangle|=|\langle x_0-y_0,\nu\rangle| \leq \tfrac{1}{2}(1- r_\e)$, \ie $|s| \leq \frac{1}{\sqrt{3}} (1-r_\e) <  \sqrt{3} \delta$. Then, using that $|\langle \tb_\alpha, \nu^\perp \rangle | \leq \frac{1}{2}$,
        \begin{equation*}
            \begin{split}
                |\langle x, \nu^\perp \rangle | & \leq |\langle x - x_0, \nu^\perp \rangle | + |\langle x_0 - y_0, \nu^\perp \rangle | + |\langle y_0, \nu^\perp \rangle |  \\
                & \leq \e + |s| |\langle \tb_\alpha, \nu^\perp \rangle | + \tfrac{1}{2}-\tfrac{5}{2}\delta < \e + \tfrac{\sqrt{3}}{2} \delta + \tfrac{1}{2}-\tfrac{5}{2}\delta < \tfrac{1}{2} - \tfrac{3}{2}\delta < \tfrac{r_\e}{2}  \, ,   
            \end{split}
        \end{equation*} 
        \ie $x \in R^\nu_{r_\e,1}$ and hence $x \in (L_\e + B_{3 \e}(0)) \cap R^\nu_{r_\e,1} = S_\e^\mathrm{top}$. We set
        \begin{equation} \label{def:tilde u 3}
            \text{if } T \subset R^\nu_{1-5\delta,1} \sm \overline Q^\nu_{r_\e} : \quad \tilde u_\e|_T :=  \begin{cases}
                u_\e^\mathrm{top}|_T & \text{if } T \subset \Sigma^{\alpha,z}_\e(T_{z}) \, , \\
                u_\e|_T & \text{otherwise.}
            \end{cases}
        \end{equation} 
        The definition is consistent with the previous case: if $T \subset Q^\nu_{r_\e + 6\e} \sm \overline Q^\nu_{r_\e}$, then $T$ is not contained in any half-slice $\Sigma^{\alpha,z}_\e(T_{z})$ (because $T_z \cap \partial  Q^\nu_{r_\e + 6\e} \neq \emptyset$) and thus $\tilde u_\e|_T = u_\e|_T$, in accordance with~\eqref{def:tilde u 1}. If $T \subset R^\nu_{1-5\delta,1} \sm \overline Q^\nu_{r_\e}$ but $T$ is not contained in any half-slice $\Sigma^{\alpha,z}_\e(T_{z})$, then $T \subset S_\e$. In particular, by~\eqref{est:on top part} and~\eqref{est:energy on slice} we infer that 
        \begin{equation} \label{est:on vertical part of P}
            F_\e(\tilde u_\e, R^\nu_{1-5\delta,1} \sm \overline Q^\nu_{r_\e})  \leq \sum_{z \in \mathcal{Z}^{\mathrm{top}}_\e} \hspace{-0.5em} F_\e(u^\mathrm{top}_\e, \Sigma^{\alpha,z}_\e(T_z))  + F_\e(u_\e, S_\e) \leq C (\sigma_\e + \e) \, .
        \end{equation}

        \noindent{\em Case} $T \subset (R^\nu_{1,1-5\delta} \sm \overline R^\nu_{1,3\delta}) \sm \overline Q^\nu_{r_\e}$ (part of the cross-shaped set $P_\delta$  aligned with $\nu^\perp$): As in the previous case, assuming $T \subset Q^\nu_+$, we define $\tilde u_\e|_T := u_\e^\mathrm{left}|_T$ if $T$ is contained in a half-slice starting from a triangle in $S_\e^\mathrm{left}$, $\tilde u_\e|_T := u_\e^\mathrm{right}|_T$ if $T$ is contained in a half-slice starting from a triangle in $S_\e^\mathrm{right}$, and $\tilde u_\e|_T := u_\e|_T$ otherwise. As before, the definition is compatible with~\eqref{def:tilde u 1}. Similarly to~\eqref{est:on vertical part of P} we obtain that 
        \begin{equation}\label{est:on horizontal part of P}
            F_\e(\tilde u_\e, (R^\nu_{1,1-5\delta} \sm \overline R^\nu_{1,3\delta}) \sm \overline Q^\nu_{r_\e})   \leq C (\sigma_\e +\e)\, .
        \end{equation}

        \noindent{\em Case} $T \cap (\R^2 \sm  P_\delta) \neq \emptyset$: let $x$ be a vertex of $T$ and assume that $x$ is not the vertex of a triangle~$T'$ covered by the previous cases. Then we set  $\tilde u_\e(x):=u_\e^\pos(x)$ if $\langle x, \nu \rangle \geq 0$ and $\tilde u_\e(x):=u_\e^\neg(x)$ if $\langle x, \nu \rangle < 0$. In particular, 
        \begin{equation}\label{def:tilde u 4} 
            \begin{aligned}
                \text{if } T \subset Q^\nu_+ \sm  \overline P_\delta : \quad &\tilde u_\e|_T = u_\e^\pos|_T \,,  \\
                \text{if } T \subset Q^\nu_- \sm  \overline P_\delta : \quad &\tilde u_\e|_T = u_\e^\neg|_T\, .  
            \end{aligned}
        \end{equation}
        We remark that 
        \begin{equation} \label{est:number of T on deP}
            \e^2 \#\{T \subset Q^\nu \colon T  \cap \partial P_\delta \neq \emptyset \}  \leq C \delta \e \quad  \implies \quad  \#\{T \subset Q^\nu \colon T  \cap \partial P_\delta \neq \emptyset \}  \leq C \frac{\delta}{\e}  
        \end{equation}
        and
        \begin{equation} \label{est:number of T on L}
            \# \{ T \subset Q^\nu \sm  \overline P_\delta  \colon  T \cap L^\nu \neq \emptyset \} \leq C \frac{\delta}{\e} \, .
        \end{equation}    
        
        Let us check that $\tilde u_\e$ attains the desired boundary conditions~\eqref{condition:bc}. Let $T \subset Q^\nu_+ \sm \overline Q^\nu_{1-\delta}$. If $T \subset  (Q^\nu_+ \sm \overline Q^\nu_{1-\delta}) \cap R^\nu_{1-5\delta,1}$ (and similarly if $T \subset  (Q^\nu_+ \sm \overline Q^\nu_{1-\delta}) \cap (R^\nu_{1,1-5\delta} \sm \overline R^\nu_{1,3\delta})$), then we are in the case covered by~\eqref{def:tilde u 3}. By~\eqref{eq:attained boundary conditions} we have $\tilde u_\e|_T = u_\e^\mathrm{top}|_T = u_\e^\pos|_T$. Otherwise, if $T \cap (\R^2 \sm P_\delta) \neq \emptyset$, let $x$ be a vertex of $T$ and assume that $x$ is not the vertex of a triangle $T'$ covered by the previous cases. Then, by definition, $\tilde u_\e(x):=u_\e^\pos(x)$.  We argue analogously if $T \subset Q^\nu_- \sm \overline Q^\nu_{1-\delta}$. Finally, if $T \cap L^\nu \neq \emptyset$, then $T \subset L^\nu + B_{2\e}(0)$ and thus it is not relevant for the boundary conditions by the definition of discrete boundary $\partial_\e^\pm Q^\nu$. 
    
    \end{step}

    \begin{step}{6} (Energy estimate). By~\eqref{est:number of T on deP}, \eqref{def:tilde u 4}, and~\eqref{est:number of T on L} we have that
        \begin{equation*}
            \begin{split}
                F_\e(\tilde u_\e, Q^\nu) & \leq F_\e(\tilde u_\e, P_\delta) + F_\e(\tilde u_\e, Q^\nu \sm \overline P_\delta) + \sum_{T  \cap \partial P_\delta \neq \emptyset} F_\e(\tilde u_\e, T) \\
                & \leq F_\e(\tilde u_\e, P_\delta) + F_\e(u_\e^\pos, Q^\nu_+ \sm \overline P_\delta) + F_\e(u_\e^\neg, Q^\nu_- \sm \overline P_\delta)  + \sum_{\substack{T \subset Q^\nu \sm  \overline P_\delta  \\ T \cap L^\nu \neq \emptyset}} F_\e(\tilde u_\e, T) + C\delta \\
                & \leq F_\e(\tilde u_\e, P_\delta) + C \delta  \, .
            \end{split}
        \end{equation*}
        Moreover, by~\eqref{def:tilde u 1}, \eqref{est:on vertical part of P}, and~\eqref{est:on horizontal part of P} we deduce that 
        \begin{equation*}
            \begin{split}
                F_\e(\tilde u_\e, P_\delta) & \leq F_\e(u_\e, P_\delta \cap  Q^\nu_{r_\e+6\e}) + F_\e(\tilde u_\e, R^\nu_{1-5\delta,1} \sm \overline Q^\nu_{r_\e}) + F_\e(\tilde u_\e, (R^\nu_{1,1-5\delta} \sm \overline R^\nu_{1,3\delta}) \sm \overline Q^\nu_{r_\e}) \\
                & \leq F_\e(u_\e,   Q^\nu ) + C  \left( \sigma_\e +\e\right) \,  .
            \end{split}
        \end{equation*}
        In conclusion, 
        \begin{equation*}
            \limsup_{\e \to 0} F_\e(\tilde u_\e, Q^\nu)  \leq \lim_{\e \to 0} F_\e(u_\e,   Q^\nu ) + C \delta \, . 
        \end{equation*}
        Eventually, letting $\delta \to 0$ and with a diagonal argument, we construct a sequence which satisfies~\eqref{condition:energy estimate}.
    \end{step}

\end{proof}

In the proof of Proposition~\ref{prop:psi is phi} we applied the following lemma. 

\begin{lemma} \label{lemma:chain of triangles}
    Let $\Sigma^{\alpha,z}_\e$ be the slices of the triangular lattice defined in~\eqref{def:slice}. Let $L$ be a line in~$\R^2$ orthogonal to~$\nu$ and assume that $|\langle \tb_\alpha , \nu^\perp \rangle | \leq \frac{1}{2}$. Then there exists a chain of triangles $(T_z)_{z \in \Z}$ satisfying for every $z \in \Z$
    \begin{equation} \label{eq:chain}
        T_z \in \T^+_\e(\R^2) \,,  \quad  T_z \subset \Sigma^{\alpha,z}_\e , \quad  T_z \cap L \neq \emptyset \, , \quad  T_z \cap T_{z+1}\neq \emptyset \, .
    \end{equation}
\end{lemma}
\begin{proof}

    It is enough to prove the following:

    \noindent {\em Claim}: Let $z \in \Z$ and let $T_z \in \mathcal{T}^+_\e(\R^2)$ be such that $T_z \subset \Sigma^{\alpha,z}_\e$ and $T_z \cap L \neq \emptyset$. Then there exists~$T_{z+1} \in \mathcal{T}^+_\e(\R^2)$ such that $T_{z+1} \subset \Sigma^{\alpha,z+1}_\e$, $T_{z+1} \cap L \neq \emptyset$, and $ T_z \cap T_{z+1}\neq \emptyset $. (The analogous statement with $\Sigma^{\alpha,z-1}_\e$ in place of $\Sigma^{\alpha,z+1}_\e$  holds true.) 
        
    \noindent With the proven claim at hand it is immediate to define a chain of triangles $(T_z)_{z \in \Z}$ which satisfies the properties in~\eqref{eq:chain} by initializing the construction from a triangle $T_{z_0} \in \T^+_\e(\R^2)$ which satisfies $T_{z_0} \cap L \neq \emptyset$ and $T_{z_0} \subset \Sigma^{\alpha,z_0}_\e$. Such a triangle always exists since the set $\R^2 \sm \bigcup_{T \in \T^+_\e(\R^2)} T$ is the union of disjoint open triangles, thus cannot contain $L$.

    To prove the claim let us denote be $\tb_\beta,\tb_\gamma$ the remaining two unit vectors   connecting points of $\L$ and   introduced in Section~\ref{subsec:lattice} and let us set $\tau_\beta\defas\sign\langle\tb_\beta,\tb_\alpha^\perp\rangle$, $\tau_\gamma\defas\sign\langle\tb_\gamma,\tb_\alpha^\perp\rangle$. For later use we observe that $\tau_\beta=\tau_\gamma$ if and only if $\alpha\neq 2$,  that is   if and only if $\langle\tb_\beta,\tb_\alpha\rangle=-\langle\tb_\gamma,\tb_\alpha\rangle$. In particular, we always have
    \begin{equation}\label{cond:sign}
    \tau_\beta\langle\tb_\beta,\tb_\alpha\rangle=-\tau_\gamma\langle\tb_\gamma,\tb_\alpha\rangle\quad\text{and}\quad\tau_\beta\langle\tb_\beta,\tb_\alpha\rangle\tau_\gamma\langle\tb_\gamma,\tb_\alpha\rangle=-1/4\, .
    \end{equation}

    Suppose now that $T_z \in \mathcal{T}^+_\e(\R^2)$, $T_z\subset\Sigma_\e^{\alpha,z}$ with $T_z\cap L\neq\emptyset$. The triangles $T^\beta\defas T_z+\e\tau_\beta\tb_\beta\in \mathcal{T}^+_\e(\R^2)$ and $T^\gamma\defas T_z+\e\tau_\gamma\tb_\gamma\in \mathcal{T}^+_\e(\R^2)$ satisfy $ T_z\cap T^\beta \neq \emptyset$ and $T_z\cap T^\gamma \neq \emptyset$. Moreover, they are contained in $\Sigma_\e^{\alpha,z+1}$. Indeed, for $x \in T_z$ we have $\langle \tb_\alpha^\perp, x + \e \tau_\beta  \tb_\beta \rangle = \langle \tb_\alpha^\perp, x \rangle + \e  \tau_\beta \langle \tb_\alpha^\perp, \tb_\beta \rangle = \langle \tb_\alpha^\perp, x \rangle + \e \frac{\sqrt{3}}{2} \in [\e \frac{\sqrt{3}}{2}  (z+1), \e \frac{\sqrt{3}}{2}  (z+2)]$, hence $T^\beta \subset \Sigma^{\alpha,z+1}_\e$ (analogously $T^\gamma \subset \Sigma^{\alpha,z+1}_\e$). \EEE

    \begin{figure}[H]
        \hspace{-2cm}
         \includegraphics{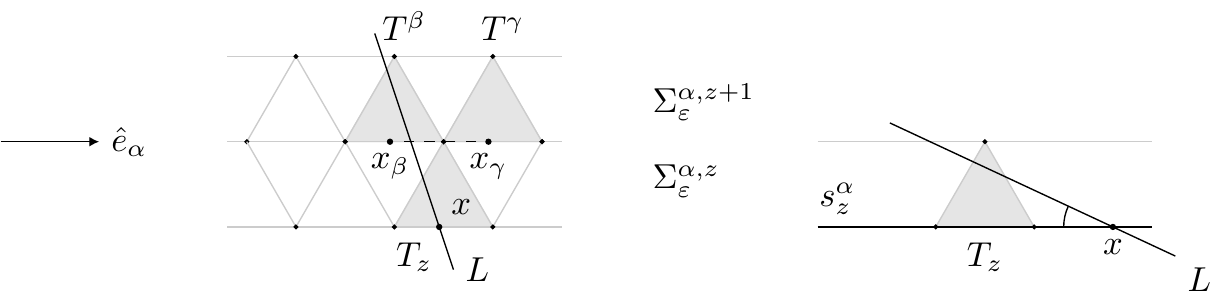}
        \caption{On the left: the triangles $T_z$, $T^\beta$ and $T^\gamma$, the line $L$, and the segment $[x_\beta;x_\gamma]$. On the right: if $x \notin T_z$, the angle between $L$ and $s_z^\alpha$ belongs to $[0,\frac{\pi}{3})$.}
        \label{fig:case2}
        \end{figure}

    The triangle $T_z$ has one side contained in $\partial \Sigma^{\alpha,z}_\e$, \ie either in $s_z^{\alpha} := \R \tb_\alpha + \e \frac{\sqrt{3}}{2} z \tb_\alpha^\perp$ or in $s_{z+1}^{\alpha} := \R \tb_\alpha + \e \frac{\sqrt{3}}{2} (z+1) \tb_\alpha^\perp$. Let us assume, without loss of generality, that the side is contained in $s_z^\alpha$. The line $L$ intersects $s_z^\alpha$ in a point $x$. We claim that $x \in T_z$. Indeed, if $x \notin T_z$, then the angle in $[0,\frac{\pi}{2}]$ between the lines $L$ and $s_z^\alpha$ belongs to $[0,\frac{\pi}{3})$, since $L$ intersects also $T_z$, see Figure~\ref{fig:case2}. Let us fix $y \in T_z \cap L \neq \emptyset$. Then we have $|\langle x-y, \tb_\alpha \rangle| > \frac{1}{2} |x-y|$. This contradicts the fact that $|\langle x-y, \tb_\alpha \rangle| = |x-y| |\langle \nu^\perp, \tb_\alpha \rangle| \leq \frac{1}{2} |x-y|$ since $|\langle \nu^\perp, \tb_\alpha \rangle| \leq \frac{1}{2}$. In conclusion $x \in T_z \cap s_z^\alpha$. Then $x_\beta\defas x+\e\tau_\beta\tb_\beta\in T^\beta \cap s_{z+1}^\alpha$, $x_\gamma\defas x+\e\tau_\gamma\tb_\gamma\in T^\gamma \cap s_{z+1}^\alpha$. The line $L$ intersects the segment $[x_\beta;x_\gamma]$, and thus either $T^\beta$ or $T^\gamma$. To see this, we let ${y_\lambda\defas \lambda x_\beta+(1-\lambda)x_\gamma}$ for $\lambda\in [0,1]$. Note that 
    \begin{equation*}
    \langle y_0-x,\nu\rangle=\e\tau_\gamma\langle\tb_\gamma,\nu\rangle=\e\tau_\gamma\big(\langle\tb_\gamma,\tb_\alpha^\perp\rangle\langle\tb_\alpha^\perp,\nu\rangle+\langle\tb_\gamma,\tb_\alpha\rangle\langle\tb_\alpha,\nu\rangle\big)=\e\Big(\frac{\sqrt{3}}{2}\langle\tb_\alpha^\perp,\nu\rangle+\tau_\gamma\langle\tb_\gamma,\tb_\alpha\rangle\langle\tb_\alpha,\nu\rangle\Big)\, ,
    \end{equation*}
     and analogously $\langle y_1-x,\nu\rangle=\e\big(\tfrac{\sqrt{3}}{2}\langle\tb_\alpha^\perp,\nu\rangle+\tau_\beta\langle\tb_\beta,\tb_\alpha\rangle\langle\tb_\alpha,\nu\rangle\big)$. In combination with \eqref{cond:sign}, this yields
    \begin{equation}\label{est:product}
    \langle y_0-x,\nu\rangle\langle y_1-x,\nu\rangle=\e^2\Big(\frac{3}{4}\langle\tb_\alpha^\perp,\nu\rangle^2-\frac{1}{4}\langle\tb_\alpha,\nu\rangle^2\Big)\leq\frac{3}{8}-\frac{3}{8}=0\, ,
    \end{equation}
    where we used that $\langle\tb_\alpha^\perp,\nu\rangle^2=\langle\tb_\alpha,\nu^\perp\rangle^2\leq\tfrac{1}{4}$ and $\langle\tb_\alpha,\nu\rangle^2\geq\tfrac{3}{4}$. Now~\eqref{est:product} together with the continuity of the mapping $\lambda\mapsto\langle y_\lambda-x,\nu\rangle$ implies that there exists $\lambda\in[0,1]$ with $\langle y_\lambda-x,\nu\rangle=0$, hence $y_\lambda\in (T^\beta \cup T^\gamma)\cap L$.

\end{proof}

\section{Upper Bound}\label{section:upperbound}
It remains to prove the $\Gamma$-limsup inequality to complete the proof of Theorem~\ref{maintheorem}.
\begin{proposition}  \label{prop:upperbound}
    Let $F_\e$ be as in~\eqref{eq:extended energy}. Then for every $\chi\in L^1(\Omega)$ we have
    \begin{align}\label{ineq: Gammalimsup}
        \Gamma\hbox{-}\limsup_{\varepsilon \to 0} F_\varepsilon(\chi) \leq F(\chi)\, ,
    \end{align}
    where $F$ is given by \eqref{def: limit energy} and the $\Gamma\hbox{-}\limsup$ is with respect to the strong topology in $L^1(\Omega)$. 
\end{proposition}
\begin{proof} It is not restrictive to assume that $\chi\in BV(\Omega;\{-1,1\})$. Moreover, thanks to Remark~\ref{rem: well-defined}, the density result \cite[Corollary 2.4]{BraConGar}, and the $L^1$-lower semicontinuity of the $\Gamma$-limsup it suffices to prove~\eqref{ineq: Gammalimsup} for $\chi \in BV(\Omega;\{-1,1\})$ such that $\js{\chi}$ is polygonal, \ie $\js{\chi}=\bigcup_{n=1}^N\Gamma_n$, where $\Gamma_n$ are line segments satisfying $\H^1(\Gamma_n\cap\partial\Omega)=0$. To simplify the exposition we restrict ourselves to the case $\js{\chi}=\Gamma_1\cup\Gamma_2$ with $\Gamma_1=[x_0;x_1]$, $\Gamma_2=[x_1;x_2]$, $x_0,x_1,x_2\in\R^2$, \ie the two segments have one common endpoint. The general case then follows by repeating the construction on each line segment.

\begin{step}{1} (Construction of a recovery sequence)
Denoting by $\ell_1,\ell_2$ the length of $\Gamma_1$, $\Gamma_2$ and by $\nu_1,\nu_2$ the outer unit normal to the set $\{\chi=1\}$ on $\Gamma_1$, $\Gamma_2$, upon relabeling we can assume that $x_1=x_0+\ell_1\nu_1^\perp$, $x_2=x_1+\ell_2\nu_2^\perp$. 
Moreover, we have
\begin{align}\label{eq:energyu}
F(\chi)=\ell_1\varphi(\nu_1)+\ell_2\varphi(\nu_2)\, ,
\end{align}
where $\varphi$ is as in \eqref{eq: varphi rho}.
Let $\rho>0$ be sufficiently small and let $u_{\e,\rho}^1\in\SF_\e$ and $u_{\e,\rho}^2\in\SF_\e$ be admissible for the minimum problems defining $\varphi(\nu_1)$, $\varphi(\nu_2)$, respectively with
\begin{align}\label{localoptimality}
\lim_{\e\to 0}F_\varepsilon(u_{\e,\rho}^1, Q^{\nu_1}_\rho) = \rho\, \varphi(\nu_1)\quad\text{and}\quad \lim_{\e\to 0}F_\e(u_{\e,\rho}^2, Q_\rho^{\nu_2})=\rho\, \varphi(\nu_2)\, .
\end{align}

We now start constructing a recovery sequence for $\chi$ by subdividing $\Gamma_1$ and $\Gamma_2$ into segments of length of order $\rho$ and suitable shifting $u_{\e,\rho}^1$, $u_{\e,\rho}^2$ along these segments. In doing so we need to leave out a small region close to the common endpoint $x_1$. Namely, denoting by  $\theta\in(0,\pi]$ the angle between $\Gamma_1$ and $\Gamma_2$ we choose $c=c(\nu_1,\nu_2)>0$ with $c\geq\frac{1}{2}+\frac{1}{2}\cot(\frac{\theta}{2})$  and we only subdivide the smaller segments $[x_0;x_1-c\rho\nu_1^\perp]$ and $[x_1+c\rho\nu_2^\perp;x_2]$ as follows. We set $M_{\e,\rho}^1\defas\lfloor\frac{\ell_1-c\rho}{\rho+5\e}\rfloor$, $M_{\e,\rho}^2\defas\lfloor\frac{\ell_2-c\rho}{\rho+5\e}\rfloor$ and we choose lattice points
\begin{align*}
&x_{m,1}^\e\in B_{2\e}\big(x_0+m(\rho+5\e)\nu_1^\perp\big)\cap\L_\e^1\quad \text{for}\ m\in\{0,\ldots,M_{\e,\rho}^1\}\, ,\\
&x_{m,2}^\e\in B_{2\e}\big(x_1+(c\rho+m(\rho+5\e))\nu_2^\perp\big)\cap\L_\e^1\quad \text{for}\ m\in\{0,\ldots,M_{\e,\rho}^2\}\, .
\end{align*} 
Note that the constant $c$ and the lattice points $x_{m,1}^\e$, $x_{m,2}^\e$ are chosen in such a way that, for $\e$ small enough,
\begin{equation*}
U_\rho\defas\bigcup_{m=0}^{M_{\e,\rho}^1} Q_\rho^{\nu_1}(x_{m,1}^\e)\cup \bigcup_{m=0}^{M_{\e,\rho}^2} Q_\rho^{\nu_2}(x_{m,2}^\e)\, 
\end{equation*}
is a union of pairwise disjoint cubes, see Figure~\ref{fig:limsup}. 
This allows us to define $u_{\e,\rho}\in\SF_\e$ by setting
\begin{equation*}
u_{\e,\rho}(x)\defas
\begin{cases}
u_{\e,\rho}^1(x-x_{m,1}^\e) &\text{if}\ x\in Q_\rho^{\nu_1}(x_{m,1}^\e) \, ,\ m\in\{0,\ldots,M_{\e,\rho}^1\}\, ,\\
u_{\e,\rho}^2(x-x_{m,2}^\e) &\text{if}\ x\in Q_\rho^{\nu_2}(x_{m,2}^\e) \, ,\ m\in\{0,\ldots,M_{\e,\rho}^2\}\, ,\\
u_\e^\pos(x) &\text{if}\ x\in\overline{\{\chi=1\}}\setminus U_\rho\, ,\\
u_\e^\neg(x) &\text{if}\ x\in\{\chi=-1\}\setminus U_\rho\, .
\end{cases}
\end{equation*}
We observe that since $x_{m,1}^\e,x_{m,2}^\e$ belong to the sublattice $\L_\e^1$, the boundary conditions satisfied by the shifted functions $u_{\e,\rho}^1(\,\cdot-x_{m,1}^\e)$, $u_{\e,\rho}^2(\,\cdot-x_{m,2}^\e)$ are compatible one with each other and with $u_\e^\pos$ and $u_\e^\neg$ on $\Omega\setminus U_\rho$. In particular, if $x\in\Omega$ is such that $\dist(x,\js{\chi})\geq\rho/2$ then $\chi(u_{\e,\rho})(x)=\chi(x)$, which implies that $\|\chi(u_{\e,\rho})-\chi\|_{L^1(\Omega)}\leq C\rho\H^1(\js{\chi})\to 0$ as $\rho\to 0$.
\end{step}

\begin{figure}[H]
    \scalebox{0.85}{
        \includegraphics{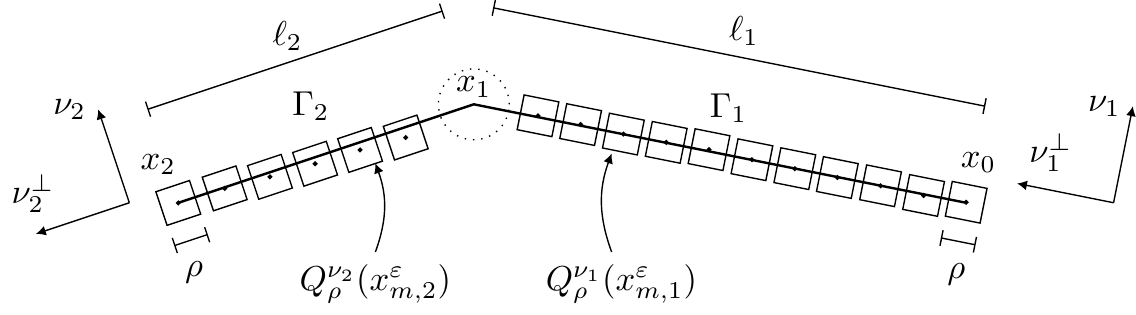}
    }
    \caption{Covering the segments $\Gamma_1$ and $\Gamma_2$ with cubes of side $\rho$ in the $\Gamma$-$\limsup$ construction.}
    \label{fig:limsup}
\end{figure}

\begin{step}{2}  (Energy estimate) 
In order to estimate $F_\e(u_{\e,\rho})$ we start by rewriting the energy as
\begin{equation}\label{eq:energy on cubes}
F_\e(u_{\e,\rho})=\sum_{m=0}^{M_{\e,\rho}^1}F_\e\big(u_{\e,\rho},Q_\rho^{\nu_1}(x_{m,1}^\e)\big)+\sum_{m=0}^{M_{\e,\rho}^2}F_\e\big(u_{\e,\rho}, Q_\rho^{\nu_2}(x_{m,2}^\e)\big)+\hspace*{-1.5em}\sum_{\substack{T\in\T_\e(\Omega)\\T\cap(\Omega\setminus U_\rho)\neq\emptyset}}\hspace*{-1.5em}F_\e(u_{\e,\rho},T)\, ,
\end{equation}
and we estimate the terms on the right-hand side of~\eqref{eq:energy on cubes} separately. Let us first consider the energy on triangles $T\in\T_\e(\Omega)$ with $T\cap(\Omega\sm U_\rho)\neq\emptyset$. Suppose that $\dist(T,\js{\chi})>5\e$. Then, if $T\subset\Omega\sm U_\rho$ we have $u_{\e,\rho}=u_\e^\pos$ or $u_{\e,\rho}=u_\e^\neg$ on $T$, so that $F_\e(u_{\e,\rho},T)=0$. If instead $T\cap U_\rho\neq\emptyset$, the fact that $\dist(T,\js{\chi})>5\e$ ensures that $T$ intersects a cube in $U_\rho$ in a region where the boundary conditions are prescribed. Thus, using once more the compatibility of the boundary conditions, we infer that $F_\e(u_{\e,\rho},T)=0$. This implies that
\begin{align}\label{est:energy outside cubes}
\sum_{\substack{T\in\T_\e(\Omega)\\T\cap(\Omega\setminus U_\rho)\neq\emptyset}}\hspace*{-1.5em}F_\e(u_{\e,\rho},T) &\leq 3\e\#\{T\in\T_\e(\Omega)\colon T\cap(\Omega\sm U_\rho)\neq\emptyset\, ,\ \dist(T,\js{\chi})\leq 5\e\}\leq C(\rho+\e/\rho)\, ,
\end{align}
where to obtain the first inequality we used $F_\e(u_{\e,\rho},T)\leq 9\e$, while the second inequality follows by counting triangles contained either in $\big([x_1-c\rho\nu_1^\perp;x_1]\cup[x_1;x_1+c\rho\nu_2^\perp]\big)+B_{6\e}(0)$ or in $\big(\partial U_\rho\cap\js{\chi}\big)+B_{6\e}(0)$.

Combining~\eqref{eq:energy on cubes},~\eqref{est:energy outside cubes}, and~\eqref{localoptimality} we deduce that
\begin{equation}\label{est:limsup}
\begin{split}
\limsup_{\e\to 0}F_\e(u_\e) &\leq\limsup_{\e\to 0}\, (M_{\e,\rho}^1+1)\, F_\e(u_{\e,\rho}^1,Q_\rho^{\nu_1})+\limsup_{\e\to 0}\, (M_{\e,\rho}^2+1)\, F_\e(u_{\e,\rho}^2,Q_\rho^{\nu_2})+C\rho\\
&\leq \Big(\Big\lfloor\frac{\ell_1}{\rho}\Big\rfloor+1\Big)\rho\, \varphi(\nu_1)+\Big(\Big\lfloor\frac{\ell_2}{\rho}\Big\rfloor+1\Big)\rho\, \varphi(\nu_2)+C\rho\, .
\end{split}
\end{equation}
Since the latter term converges to $\ell_1\varphi(\nu_1)+\ell_2\varphi(\nu_2)$ as $\rho\to 0$, thanks to~\eqref{eq:energyu} and~\eqref{est:limsup}, a diagonal argument provides us with a sequence $(u_\e)=(u_{\e,\rho(\e)})$ with $\chi(u_\e)\to\chi$ in $L^1(\Omega)$ and satisfying $\limsup_\e F_\e(u_\e)\leq F(\chi)$, from which we finally deduce \eqref{ineq: Gammalimsup}.
\end{step}
\end{proof}

\noindent {\bf Acknowledgments.}  The work of A.\ Bach and M.\ Cicalese was supported by the DFG Collaborative Research Center TRR 109, ``Discretization in Geometry and Dynamics''. G.\ Orlando has received funding from Alexander von Humboldt Foundation and the European Union's Horizon 2020 research and innovation programme under the Marie Sk\l odowska-Curie grant agreement No 792583. The work of L.\ Kreutz was funded by the Deutsche Forschungsgemeinschaft (DFG, German Research Foundation) under Germany's Excellence Strategy EXC 2044 -390685587, Mathematics M\"unster: Dynamics--Geometry--Structure.

\end{document}